\documentclass[11pt]{article}
\usepackage[T1]{fontenc}
\usepackage{lmodern}
\usepackage[margin=1in]{geometry}
\usepackage{amsmath,amssymb,amsthm,mathtools}
\usepackage{microtype}
\usepackage{enumitem}
\usepackage[colorlinks=true,linkcolor=blue,citecolor=blue,urlcolor=blue]{hyperref}
\usepackage{fancyhdr}
\newtheorem{theorem}{Theorem}[section]
\newtheorem{lemma}[theorem]{Lemma}
\newtheorem{proposition}[theorem]{Proposition}
\newtheorem{corollary}[theorem]{Corollary}
\theoremstyle{definition}
\newtheorem{remark}[theorem]{Remark}

\newcommand{\FPD}{\operatorname{FPD}}
\newcommand{\FFD}{\operatorname{FFD}}
\newcommand{\pd}{\operatorname{pd}}
\newcommand{\fd}{\operatorname{fd}}
\newcommand{\wgd}{\operatorname{w.gl.dim}}
\newcommand{\gld}{\operatorname{gl.dim}}
\newcommand{\Ann}{\operatorname{Ann}}
\newcommand{\Jac}{\operatorname{Jac}}
\newcommand{\Spec}{\operatorname{Spec}}
\newcommand{\Max}{\operatorname{Max}}
\newcommand{\depth}{\operatorname{depth}}
\newcommand{\Tor}{\operatorname{Tor}}
\newcommand{\Ext}{\operatorname{Ext}}
\newcommand{\N}{\mathbb N}
\newcommand{\Q}{\mathbb Q}
\newcommand{\F}{\mathbb F}
\newcommand{\R}{\mathbb R}
\newcommand{\ideal}{\ltimes}
\title{Big Finitistic Dimensions of Formal Power Series Rings}
\author{Xiaolei Zhang \thanks{School of Mathematics and Statistics, Tianshui Normal University, 741001 Tianshui, P. R. China} \and Weiheng Xia \thanks{School of Mathematical Sciences, Capital Normal University, 100048 Beijing, P. R. China}}
\date{}
\begin{document}
	\maketitle
	\vspace{-2.5em}
	\begin{abstract}
In this paper, we completely investigate what the big finitistic projective dimension (resp., big finitistic flat dimension) pair, which abbreviates FPD pair (resp., FFD pair), of a commutative ring and of its formal power series ring can be. When the ring is Noetherian, both invariants increase by one. In general, for every \(n\ge0\) and every \(d\) with \(n+1\le d\le\infty\), the FFD pair \((n,d)\) can be realized; the FPD pairs  \((n,d)\) can be realized when \(n\ge1\), also for \((0,1)\) and \((0,\infty)\). However, the FPD pair \((0,2)\) is impossible. For perfect coefficients, we prove  that the existence of a  commutative perfect ring \(R\) with \(\FPD(R[[x]])=3\) is equivalent to the continuum hypothesis. Under the continuum hypothesis, every finite integer \(n\ge2\) admits a  commutative local perfect ring \(R\) with \(\FFD(R[[x]])=1\) and \(\FPD(R[[x]])=n+2\). And for each such \(n\),  the existence of a perfect coefficient ring with projective value \(n+2\) does not imply the continuum hypothesis.
	\end{abstract}
	
	\noindent\textit{2020 Mathematics Subject Classification.}  13D05, 13F25.
	
	\noindent\textit{Keywords.} big finitistic projective dimension; big finitistic flat dimension; formal power series; idealization; perfect ring.
	
	\section{Introduction}
	Throughout, all rings are commutative rings with $1\not=0$, and all modules are unital. Let $R$ be a ring, and write
	\[
	\FPD(R)=\sup\{\pd_R M:\pd_R M<\infty\},\qquad
	\FFD(R)=\sup\{\fd_R M:\fd_R M<\infty\}.
	\]	
	For the power series ring \(R[[x]]\), multiplication by \(x\) is injective and \(R[[x]]/xR[[x]]=R\). Hence the regular-element change-of-rings theorem gives
	\begin{equation}\label{intro-lower}
		\FPD(R[[x]])\ge\FPD(R)+1,\qquad \FFD(R[[x]])\ge\FFD(R)+1.
	\end{equation}
The reverse inequalities are more delicate. 	Our first result says that both lower bounds in \eqref{intro-lower} are equalities in the Noetherian case.
	\begin{theorem}\label{main-noeth}
		If \(R\) is Noetherian, then
		\[
		\FPD(R[[x]])=\FPD(R)+1,\qquad
		\FFD(R[[x]])=\FFD(R)+1.
		\]
	\end{theorem}

As to coherent rings, J{\o}ndrup and Small proved the weak global dimension formula for the case when the power series ring is coherent \cite{JS}. Thus, if \(R[[x]]\) is coherent and \(\wgd R\) is finite, the flat finitistic equality follows. However, if \(\wgd R\) is not assumed finite, coherence of \(R[[x]]\) is not enough. For example, take \(R=V\ideal V/tV\), where \(V\) is the valuation ring of the full Hahn field \(\Q((t^{\R}))\). Utilizing Wang's structure properties of \(R[[x]]\)  \cite{YWang}, we can obtain the following finitistic dimensions of both \(R\) and \(R[[x]]\):
	\begin{theorem}\label{main-coherent}
The commutative ring \(R\) satisfies that \(R[[x]]\) is coherent and
		\[
		\FFD(R)=0,\qquad \FFD(R[[x]])=2,
		\qquad \FPD(R)=1,\qquad \FPD(R[[x]])\ge3.
		\]
	\end{theorem}
	
For more general commutative rings, we  have the following realizations of FPD pairs and FFD pairs.
	\begin{theorem}\label{main-pairs}
		Let \(n\in\N_0\) and \(d\in\N\cup\{\infty\}\) satisfy \(d\ge n+1\).
		\begin{enumerate}[label=\textup{(\roman*)},nosep]
			\item There is a commutative ring \(S_{n,d}\) with
			\[
			\FFD(S_{n,d})=n,\qquad \FFD(S_{n,d}[[x]])=d.
			\]
			\item If \(n\ge1\), or if \((n,d)=(0,1)\), or if \(d=\infty\), there is a commutative ring \(R_{n,d}\) with
			\[
			\FPD(R_{n,d})=n,\qquad \FPD(R_{n,d}[[x]])=d.
			\]
			\item When \(d=\infty\), the same ring \(R_{n,\infty}=S_{n,\infty}\) can be chosen so that
			\[
			\FFD(R_{n,\infty})=\FPD(R_{n,\infty})=n,\qquad
			\FFD(R_{n,\infty}[[x]])=\FPD(R_{n,\infty}[[x]])=\infty.
			\]
		\end{enumerate}
	\end{theorem}
	The above finite construction is based on the coefficient filtration of Tran, Vo and Nguyen \cite{TVN}. We separate their structural theorem from the homological computations below. A cyclic idealization gives the finite flat seed of value zero. To achieve the infinite value of the finitistic dimensions of its power series ring, we use a perfect local algebra formed by adding a common identity to a direct sum of radicals of truncated polynomial algebras.  Finally,  product with the classical case prescribes the coefficient value in both constructions.
	
	The FPD case with coefficient value zero has a different answer. By Bass's theorem, a commutative ring satisfies \(\FPD(R)=0\) exactly when it is perfect, that is, a finite product of local rings whose maximal ideals are \(T\)-nilpotent \cite{Bass,BM}.
	\begin{theorem}\label{main-perfect}
		Let \(R\) be a commutative perfect ring.
		\begin{enumerate}[label=\textup{(\roman*)},nosep]
			\item If \(\Jac(R)\) is nilpotent, then \(\FPD(R[[x]])=\FFD(R[[x]])=1\).
			\item If \(\Jac(R)\) is not nilpotent, there is an \(R[[x]]\)-module of projective dimension exactly three. Consequently \(\FPD(R[[x]])\ge3\).
		\end{enumerate}
	\end{theorem}

The last section continues to investigate the FPD pairs of perfect rings which is found to have a close connection with set-theoretic hypothesis. We always abbreviate  the Continuum Hypothesis to $\mathrm{CH}$, i.e., $2^{\aleph_0}=\aleph_1$. 
\begin{theorem}\label{main-settheory}
\begin{enumerate}[label=\textup{(\roman*)}]
\item For every finite integer $n\ge1$, 
the following three conditions are equivalent:
\begin{enumerate}[label=\textup{(\alph*)}]
\item There exists a  commutative perfect ring $R$ such that
$\FPD(R[[x]])=n+2$.
\item There exists a  commutative local perfect ring $R$ of
cardinality $\aleph_n$ such that
\[
\FFD(R[[x]])=1,\qquad \FPD(R[[x]])=n+2.
\]
\item $2^{\aleph_0}\le\aleph_n$.
\end{enumerate}
Consequently, It is equivalent to $\mathrm{CH}$  that there exists a  commutative perfect ring $R$ with
$\FPD(R[[x]])=3$, and there exists a  commutative perfect ring $R$ with
$\FPD(R[[x]])=n+2$ under 
$\mathrm{CH}$ for every finite integer $n\ge2$.
\item For every fixed finite integer $n\ge2$, it  is independent of ZFC that there exists a  commutative perfect ring $R$ with
$\FPD(R[[x]])=n+2$.  Furthermore, the property that  
``there exists a  commutative perfect ring $R$ with
$\FPD(R[[x]])=n+2$'' does not imply $\mathrm{CH}$ for every finite integer $n\ge2$.
\end{enumerate}
\end{theorem}

	\section{Regular elements and idealizations}
 In this section, we collect the homological arguments that will be used in the constructions. A regular element is a non-zero-divisor of the ring; regularity on a module is specified separately.
	
	\begin{lemma}\label{shift}\cite[Theorem 3.8.12,Theorem 3.8.14]{WK}
		Let \(c\) be a regular nonunit of a ring \(T\), put \(\overline T=T/cT\), and let \(M\ne0\) be an \(\overline T\)-module. If the dimension over \(\overline T\) is finite, then
		\[
		\pd_T M=\pd_{\overline T}M+1,\qquad
		\fd_T M=\fd_{\overline T}M+1,
		\]
		respectively.
	\end{lemma}
		
	
	\begin{lemma}\label{localization}
		For a localization \(U\) of \(T\) and a \(U\)-module \(M\),
		\(\fd_T M=\fd_U M\). If \(U=T_c\), then \(\pd_T U\le1\) and
		\[
		\pd_U M\le\pd_T M\le\pd_U M+1
		\]
		whenever \(\pd_U M<\infty\).
	\end{lemma}
	\begin{proof}
		The flat equality follows by restricting and localizing flat resolutions. On \(F=\bigoplus_{j\ge0}Te_j\), the map \(e_j\mapsto e_j-ce_{j+1}\) is injective: comparing coordinates of a finite-support vector proves this even when \(c\) is a zerodivisor. Its cokernel is \(T_c\). Hence
		\begin{equation}\label{telescope}
			0\longrightarrow F\longrightarrow F\longrightarrow T_c\longrightarrow0
		\end{equation}
		is a projective resolution. Restrict a finite \(T_c\)-projective resolution for the upper bound, and localize a \(T\)-projective resolution for the lower bound.
	\end{proof}
	
	\begin{lemma}\label{flat-gluing}
		Suppose \(c\) is regular on both a ring \(T\) and a \(T\)-module \(L\). Then \(L\) is flat if and only if \(L/cL\) is flat over \(T/cT\) and \(L_c\) is flat over \(T_c\).
	\end{lemma}
	\begin{proof}
		Necessity follows by tensoring. For sufficiency, the resolution of \(T/cT\) gives \(\Tor_i^T(T/cT,L)=0\) for \(i>0\). Change of rings and flatness of \(L/cL\) then give \(\Tor_i^T(Q,L)=0\) for every module \(Q\) killed by \(c\) and every \(i>0\). Filtering by powers of \(c\) extends this to modules killed by a power of \(c\). Since Tor commutes with filtered colimits, it extends to all \(c\)-power torsion modules.
		
		For arbitrary \(Q\), the kernel and cokernel of \(Q\to Q_c\) are \(c\)-power torsion. Also \(\Tor_i^T(Q_c,L)=\Tor_i^{T_c}(Q_c,L_c)=0\) for \(i>0\). Factor this map through its image and use the two Tor exact sequences. They give \(\Tor_1^T(Q,L)=0\).
	\end{proof}
	
	\begin{theorem}\label{flat-formula}
		For a regular nonunit \(c\) of a ring \(T\),
		\[
		\FFD(T)=\max\{\FFD(T/cT)+1,\FFD(T_c)\}.
		\]
		In particular, for every ring \(R\),
		\begin{equation}\label{intro-formula}
			\FFD(R[[x]])=\max\{\FFD(R)+1,\FFD(R((x)))\}.
		\end{equation}
	\end{theorem}
	\begin{proof}
		Lemmas \ref{shift} and \ref{localization} give the lower bounds. If their maximum is infinite, there is nothing left to prove. Otherwise set
		\(a=\FFD(T/cT)\), \(b=\FFD(T_c)\) and \(h=\max\{a+1,b\}\).
		
		Let \(M\) have finite flat dimension and take a free resolution. Its first syzygy \(Z\) is \(c\)-regular. Reduction of its finite flat resolution modulo \(c\) is exact, so \(\fd_{T/cT}(Z/cZ)\le a\). The module \(L=\Omega_T^hM\) is \(c\)-regular, and \(L/cL\) is an \((h-1)\)st syzygy of \(Z/cZ\), hence flat. The localized module \(L_c\) is flat because \(\fd_{T_c}M_c\le b\le h\). Lemma \ref{flat-gluing} makes \(L\) flat.
	\end{proof}
	
	For a ring \(H\) and an \(H\)-module \(E\), the idealization \(H\ideal E\) has underlying module \(H\oplus E\) and multiplication \((a,e)(b,f)=(ab,af+be)\).
	
	\begin{theorem}\label{idealization}
		Let \(s\) be a regular nonunit of a ring \(H\), let \(E=H\ideal H/sH\), and put \(J=0\ideal H/sH\). If \(\fd_E M<\infty\) and \(N=M/JM\), then
		\[
		\Tor_i^E(H,M)=\Tor_i^E(J,M)=0\quad(i>0),\qquad
		s:N\longrightarrow N\text{ is injective},\qquad
		\fd_E M=\fd_H N.
		\]
		If \(\pd_E M<\infty\), then also \(\pd_E M=\pd_H N\). Conversely, if \(s\) is regular on an \(H\)-module \(N\) of finite flat or projective dimension, \(E\otimes_H N\) has that same dimension over \(E\). Consequently
		\begin{align}
			\FFD(E)&=\sup\{\fd_H N<\infty:s\text{ is regular on }N\},\label{ideal-flat}\\
			\FPD(E)&=\sup\{\pd_H N<\infty:s\text{ is regular on }N\}.\label{ideal-proj}
		\end{align}
	\end{theorem}
	\begin{proof}
		Inflate \(H\)-modules along \(E\to H\). There are exact sequences
		\[
		0\to J\to E\to H\to0,\qquad
		0\to H\oplus J\to E\to J\to0,\qquad
		0\to H\xrightarrow{s}H\to J\to0.
		\]
		In the second sequence the surjection is multiplication by \((0,1)\); its kernel \(sH\oplus J\) is isomorphic to \(H\oplus J\) as an \(E\)-module, since \(sJ=0\). Writing \(T_i=\Tor_i^E(H,M)\) and \(U_i=\Tor_i^E(J,M)\) gives
		\[
		T_{i+1}\cong U_i,\qquad U_{i+1}\cong T_i\oplus U_i\quad(i\ge1).
		\]
		They vanish in high degrees, and the second recurrence forces their vanishing in every positive degree. The third sequence then proves regularity of \(s\) on \(N\).
		
		A free resolution of \(M\) reduces to a free \(H\)-resolution of \(N\), and every syzygy has the same positive Tor vanishing. Flatness lifts across a nilpotent ideal when the reduction is flat and \(\Tor_1^E(H,-)=0\) \cite[Tag 051C]{Stacks}. Thus the corresponding syzygies are flat at exactly the same stages, proving the flat equality. A flat module whose reduction modulo a nilpotent ideal is projective is projective \cite[Tag 05CG]{Stacks}. Apply this to the \(\pd_H N\)th syzygy to obtain the projective equality.
		
		Conversely \(E=H\oplus H/sH\) as an \(H\)-module, so regularity of \(s\) on \(N\) gives \(\Tor_i^H(E,N)=0\) for \(i>0\). Tensoring a finite flat or projective resolution is exact, and its reduction modulo \(J\) is \(N\). The equalities already proved show that its dimension is unchanged.
	\end{proof}
	
	\begin{corollary}\label{ideal-bounds}
		In the notation of Theorem \ref{idealization},
		\[
		\FFD(H_s)\le\FFD(E)\le\max\{\FFD(H/sH),\FFD(H_s)\},
		\qquad \FPD(E)\le\FPD(H).
		\]
	\end{corollary}
	\begin{proof}
		An \(H_s\)-module is \(s\)-regular, so the lower bound follows from Lemma \ref{localization} and \eqref{ideal-flat}. For an admissible \(N\) in that formula, both \(N/sN\) and \(N_s\) have finite flat dimension. If the displayed maximum is \(h<\infty\), its \(h\)th free syzygy is flat modulo \(s\) and after localization. Lemma \ref{flat-gluing} applies; when \(h=0\), apply it to \(N\) itself. The projective bound follows from \eqref{ideal-proj}.
	\end{proof}
	
	We shall also use two results on lifting projective dimension. If \(T_s\) and every \(T/s^jT\) are perfect, then
	\begin{equation}\label{almost-perfect}
		\fd_T M\le1\ \Longrightarrow\ \pd_T M\le1.
	\end{equation}
	This is \cite[Proposition 8.2]{BP}. If \(F\) is flat, \(F_s\) is projective, and \(F/s^jF\) is projective over \(T/s^jT\) for every \(j\ge1\), then
	\begin{equation}\label{strong-flat}
		\pd_T F\le1.
	\end{equation}
	Indeed \cite[Corollary 9.7]{Pos} makes \(F\) strongly flat for the countable localization: it is a summand of an extension of a free \(T_s\)-module by a free \(T\)-module. Equation \eqref{telescope} gives \eqref{strong-flat}. These statements require no finiteness condition on the module.
	
	\section{Noetherian rings}
In this section, we will investigate the FPD pairs and FFD pairs of Noetherian rings.  For a commutative Noetherian ring \(T\), the classical big finitistic dimension formulas are
	\begin{equation}\label{classical}
		\FPD(T)=\dim T,\qquad \FFD(T)=\sup_{\mathfrak p\in\Spec T}\depth T_{\mathfrak p}.
	\end{equation}
	The flat formula is due to Auslander and Buchsbaum, and the projective formula combines the lower bound of Bass with the upper bound of Gruson and Raynaud. Both are recorded in \cite[Theorems 3.1 and 3.2]{Foxby}. For a Noetherian local ring \(T\) of dimension \(e\), the same results give
	\begin{equation}\label{local-classical}
		\FFD(T)=
		\begin{cases}
			e,&T\text{ is Cohen--Macaulay},\\
			e-1,&T\text{ is not Cohen--Macaulay}.
		\end{cases}
	\end{equation}
	In particular, the flat invariant in a local ring is not generally its depth.
	
	\begin{lemma}\label{flat-local}\cite[Proposition 2.2]{WKX}
		For every commutative ring \(T\),
		\[
		\FFD(T)=\sup_{\mathfrak m\in\Max T}\FFD(T_{\mathfrak m}).
		\]
	\end{lemma}

	\begin{proof}[Proof of Theorem \ref{main-noeth}]
		Put \(A=R[[x]]\). It is Noetherian, and \(x\in\Jac(A)\), since \(1-xa\) has a formal inverse for every \(a\in A\). Its maximal ideals are therefore \((\mathfrak m,x)\), with \(\mathfrak m\in\Max R\). For
		\[
		T=A_{(\mathfrak m,x)},\qquad V=R_{\mathfrak m},
		\]
		the element \(x\) is regular in the maximal ideal and \(T/xT\cong V\). The dimension and depth formulas give
		\[
		\dim T=\dim V+1,\qquad \depth T=\depth V+1.
		\]
		Thus \(T\) is Cohen--Macaulay if and only if \(V\) is. Each Noetherian local ring has finite dimension, so \eqref{local-classical} gives
		\(\FFD(T)=\FFD(V)+1\) in either case. Lemma \ref{flat-local} proves the flat equality.
		
		Likewise
		\[
		\dim A=\sup_{\mathfrak m}\dim A_{(\mathfrak m,x)}
		=\sup_{\mathfrak m}(\dim R_{\mathfrak m}+1)=\dim R+1.
		\]
		Now use \eqref{classical}. If the Krull dimension supremum is infinite, these equalities mean that the local dimensions are unbounded; the same arguments still apply.
	\end{proof}
	
	\section{Coherent formal power series rings}
In this section we show that, even for a commutative ring $R$ with $R[[x]]$ coherent,
the FPD and FFD of $R[[x]]$ need not follow the clean rule they obey in the Noetherian case:
both invariants strictly increase by one. Coherence alone is not enough to force that equality.

	Let \(K=\Q((t^{\R}))\) be the full Hahn field, with all well-ordered supports allowed, and let \(v\) be its valuation. In this section, we always set
	\[
	V=\{a\in K:v(a)\ge0\},\qquad A=V[[x]],\qquad
	R=V\ideal V/tV,\qquad B=R[[x]].
	\]
	The full Hahn field is spherically complete, with value group \(\R\). Wang's theorem \cite[Theorem 1.1(i) and Remark 1.2]{YWang} makes \(A\) coherent. The coherent power series theorem \cite{JS} then gives
	\begin{equation}\label{hahn-wgd}
		\wgd A=\wgd V+1=2.
	\end{equation}
	Principal coefficient division identifies
	\[
	B\cong A\ideal A/tA.
	\]
	
	\begin{lemma}\label{valuation-pd}
		Every torsion-free \(V\)-module has projective dimension at most one.
	\end{lemma}
	\begin{proof}
		For a nonzero submodule \(L\subseteq K\), its values form an upper subset of \(\R\), with a countable coinitial subset. Consequently \(L\) is a countable increasing union of principal fractional ideals. These are free of rank one, so the telescope gives \(\pd_V L\le1\).
		
		For a torsion-free module \(M\), embed it in \(K\otimes_VM\) and well-order a vector-space basis. Intersect \(M\) with the spans of initial segments. This gives a continuous filtration whose successive quotients embed in \(K\) by their new coordinate and have projective dimension at most one. At successors, the horseshoe lemma joins their length-one resolutions with split inclusions in the resolution terms. At limits, take filtered colimits. The terms are direct sums of the added projectives and the sequences remain exact, giving a length-one resolution of \(M\).
	\end{proof}
	
	\begin{proposition}\label{hahn-upper}
		One has \(\FFD(R)=0\), \(\FPD(R)=1\). The ring \(B\) is coherent and \(\FFD(B)\le2\).
	\end{proposition}
	\begin{proof}
		Theorem \ref{idealization} reduces a finite-flat-dimensional \(R\)-module to a \(t\)-regular \(V\)-module. Every nonzero \(a\in V\) divides some \(t^n\), since the value group is Archimedean. Such a module is torsion-free and hence flat over \(V\) \cite[Tag 0539]{Stacks}. Thus \(\FFD(R)=0\). Lemma \ref{valuation-pd} gives \(\FPD(R)\le1\). The localization \(R_t=K\) has projective dimension exactly one: \eqref{telescope} gives the upper bound, and it cannot be free over the local ring \(R\), since the nonunit \(t\) acts surjectively. Projective modules over local rings are free \cite{Kaplansky}.
		
		As an \(A\)-module, \(B=A\oplus A/tA\) is finitely presented. A map between finite free \(B\)-modules is a map between finitely presented \(A\)-modules. Its kernel is finitely generated over coherent \(A\), and therefore over \(B\). This proves coherence. Equation \eqref{hahn-wgd} and Theorem \ref{idealization} give the flat upper bound.
	\end{proof}
	
	We now construct the lower-bound modules. Put
	\[
	\lambda_i=2^{-i},\qquad a_i=t^{\lambda_i},\qquad b_i=a_i+t^i,
	\]
	and define coefficientwise by Hahn summation
	\begin{equation}\label{hahn-products}
		f=t\prod_{i\ge1}(1-x/a_i),\qquad
		h=t\prod_{i\ge1}(1-x/b_i).
	\end{equation}
	These are not \(x\)-adic limits of the partial products.
	
	\begin{lemma}\label{hahn-sums}
		The expressions \eqref{hahn-products} belong to \(A\). For \(N\ge0\), put \(\delta_N=2^{-N}\). The normalized tails
		\[
		F_N=t^{\delta_N}\prod_{i>N}(1-x/a_i),\qquad
		H_N=t^{\delta_N}\prod_{i>N}(1-x/b_i)
		\]
		belong to \(A\) and have constant term \(t^{\delta_N}\). Moreover,
		\[
		f(a_i)=0,\qquad v(h(a_i))=i+i2^{-i}.
		\]
	\end{lemma}
	\begin{proof}
		For fixed \(r\), the \(x^r\) coefficient of \(f\) is the Hahn sum
		\[
		(-1)^r\sum_{\substack{F\subseteq\N_{>0}\\|F|=r}}
		t^{\,1-\sum_{i\in F}2^{-i}}.
		\]
		The exponents are distinct and well ordered, with least element \(2^{-r}\). Indeed, order the increasing \(r\)-tuples lexicographically. At the first different index, its binary weight exceeds the sum of any finite collection of later weights. This gives exactly the exponent order. All coefficients therefore lie in \(V\).
		
		For \(h\), set \(\eta_i=i-2^{-i}>0\) and expand
		\[
		b_i^{-1}=t^{-2^{-i}}\sum_{m\ge0}(-1)^m t^{m\eta_i}.
		\]
		Its \(x^r\) coefficient has exponents
		\(1-\sum_{i\in F}2^{-i}+\sum_{i\in F}m_i\eta_i\), \(|F|=r\).
		Below any fixed real bound there are only finitely many nonzero correction patterns: a positive \(m_i\) forces \(\eta_i\) to be bounded, and \(\eta_i\ge1/2\) bounds \(m_i\). For a fixed correction pattern, the remaining exponents are a subset of a translate of the well-ordered set already considered. Each exponent has finite multiplicity. This verifies Hahn summability. Its least exponent is again \(2^{-r}\), attained uniquely without corrections.
		
		The same argument with indices \(i>N\) proves the assertions for the tails. Coefficientwise multiplication gives
		\[
		f=\prod_{j=1}^N(a_j-x)F_N,\qquad
		h=\prod_{j=1}^N\frac{t^{\lambda_j}}{b_j}(b_j-x)H_N.
		\]
		A series in \(A\) can be evaluated at any \(z\) of positive valuation: its terms have valuations tending to infinity. If \(v(z)>\delta_N\), the constant term of \(H_N(z)\) dominates, so its valuation is \(\delta_N\). Fix \(i\) and take \(N>i\). The factorization gives \(f(a_i)=0\). Also
		\[
		v(b_j-a_i)=
		\begin{cases}
			\lambda_i,&j<i,\\
			i,&j=i,\\
			\lambda_j,&j>i.
		\end{cases}
		\]
		The factors \(t^{\lambda_j}/b_j\) have valuation zero. Summing gives
		\((i-1)\lambda_i+i+\sum_{j=i+1}^N\lambda_j+\delta_N=i+i\lambda_i\).
	\end{proof}
	
	\begin{lemma}\label{hahn-regular}
		If \(q\in A\) vanishes at every \(a_i\), then \(q\in fA\). Consequently \(f,h\) is an \(A\)-regular sequence, and \((f,h)A_t\) is proper.
	\end{lemma}
	\begin{proof}
		A root can be divided out inside \(A\). For \(q=\sum q_jx^j\) and \(q(a)=0\) with \(v(a)>0\), the coefficients
		\[
		Q_j=-\sum_{m\ge j+1}q_ma^{m-j-1}\in V
		\]
		give \(q=(a-x)\sum Q_jx^j\). Every other distinct root remains a root of the quotient. Hence \(q/P_N\in A\), where \(P_N=\prod_{i=1}^N(a_i-x)\).
		
		Let \(z=q/f\in K[[x]]\). Since \(f=P_NF_N\),
		\(z=(q/P_N)F_N^{-1}\). The inverse recurrence yields
		\(v([x^j]F_N^{-1})\ge-(j+1)\delta_N\). Thus
		\(v([x^j]z)\ge-(j+1)\delta_N\) for all \(N\). Letting \(N\to\infty\) gives \(z\in A\).
		
		The ring \(A\) is a domain and \(f\ne0\). If \(hq\in fA\), evaluation at \(a_i\) gives \(q(a_i)=0\), since \(h(a_i)\ne0\). Thus \(q\in fA\), proving regularity. The ideal is proper in local \(A\), since both constant terms are \(t\).
		
		Suppose \(1=pf+qh\) in \(A_t\). The coefficients of \(q\) have a common valuation lower bound \(-m\). Evaluation at \(a_i\) gives
		\(1=q(a_i)h(a_i)\), but its right side has valuation at least
		\(-m+i+i2^{-i}>0\) for large \(i\), a contradiction.
	\end{proof}
	
	\begin{proof}[Proof of Theorem \ref{main-coherent}]
		The coefficient assertions and coherence are Proposition \ref{hahn-upper}. Put \(D=A_t\) and \(N=D/(f,h)D\ne0\). The Koszul resolution gives \(\fd_DN=\pd_DN=2\); for the lower bound, its tensor with \(N\) has zero differentials and nonzero top homology. Since \(B_t=D\), Lemma \ref{localization} gives \(\FFD(B)\ge2\), proving equality.
		
		For the projective assertion, set \(C=A/(f,h)\). The localization \(C_t=N\) has \(C\)-projective dimension one: the telescope gives the upper bound, and it cannot be free over local \(C\) because \(t\) is a nonunit acting surjectively. Applying Lemma \ref{shift} twice gives \(\pd_A N=3\). The module \(N\) is \(t\)-regular, and \((A\ideal A/tA)\otimes_A N=N\). Theorem \ref{idealization} gives \(\pd_B N=3\).
	\end{proof}
	
	\section{The existence of pairs of FFDs and FPDs}
In this section, we will investigate what the FPD and FFD pairs  of a general commutative ring can be.
	\subsection{A coefficient-filtered seed}
	Fix \(r\ge2\), put \(k=\Q\) and \(K=k(u_1,\ldots,u_r)\). Enumerate representatives \(q_1,q_2,\ldots\) of the associate classes of irreducibles of \(k[u_1,\ldots,u_r]\). Define
	\[
	D_N=\prod_{j=1}^Nq_j^{\lfloor N/j\rfloor},\qquad
	W_N=\left\{\frac p{D_N}:\deg p\le\deg D_N+N\right\},\qquad D_0=1.
	\]
	Total degrees are taken in the \(u_i\). These are finite-dimensional \(k\)-spaces, \(W_0=k\), and
	\[
	W_N\subseteq W_{N+1},\qquad W_NW_M\subseteq W_{N+M},
	\qquad \bigcup_NW_N=K.
	\]
	The product inclusion follows from
	\(\lfloor N/j\rfloor+\lfloor M/j\rfloor\le\lfloor(N+M)/j\rfloor\).
	After a common denominator is chosen, the numerator bound increases by precisely its degree. Exhaustion follows by choosing \(N\) large enough to absorb the denominator and excess numerator degree.
	
	In this section, we always set
	\begin{equation}\label{filtered-T}
		T=T_r=\left\{\sum_{j\ge0}a_jc^j\in K[[c]]:
		a_j\in W_{2^j-1}\right\},\qquad A=T[[x]].
	\end{equation}
	We use the following structural theorem of the coefficient-filtered construction \cite[Sections 2--4]{TVN}. Write
	\[
	F=k[[x]],\quad
	B_0=(k[u_1,\ldots,u_r]\setminus\{0\})^{-1}F[u_1,\ldots,u_r],
	\quad E_N=W_N\otimes_kF.
	\]
	For \(b\in B_0\) let \(w(b)=\min\{N:b\in E_N\}\), and put
	\[
	H=\left\{\sum_{j\ge0}b_jc^j\in B_0[[c]]:
	w(b_j)\le M2^j\text{ for all }j\text{ and some }M\right\}.
	\]
	The cited bounded-division and Noetherianity results state that
	\begin{equation}\label{structural}
		H\text{ is Noetherian},\quad c\in\Jac(H),\quad
		H/cH=B_0,\quad \dim H_c\le r.
	\end{equation}
	We treat \eqref{structural} as the structural input; the finitistic computations are proved here. Its completion is \(B_0[[c]]\): coefficient shifts show
	\(H\cap c^eB_0[[c]]=c^eH\), and every finite \(c\)-polynomial belongs to \(H\).
	
	\begin{lemma}\label{T-properties}
		The ring \(T\) is local, with maximal ideal \(\mathfrak m\) consisting of the series with zero constant term, and \(T/\mathfrak m=k\). Moreover
		\[
		\mathfrak m^3\subseteq cT,\qquad T_c\text{ is a field},\qquad A_c=H_c.
		\]
	\end{lemma}
	\begin{proof}
		The weight is subadditive under multiplication. For positive \(j_1+\cdots+j_s=n\),
		\(\sum_\nu(2^{j_\nu}-1)\le2^n-1\). The usual inverse recurrence therefore stays within the coefficient bounds of \eqref{filtered-T} whenever the constant coefficient is nonzero. This proves locality.
		
		For positive \(i+j+\ell=n\),
		\[
		(2^i-1)+(2^j-1)+(2^\ell-1)\le2^{n-1}-1.
		\]
		The largest sum of powers occurs at \((n-2,1,1)\). This proves \(\mathfrak m^3\subseteq cT\) by dividing a triple product coefficientwise by \(c\).
		
		Any series whose coefficient weights are at most \(M2^j\) enters \(T\) after multiplication by a sufficiently large power of \(c\). If \(0\ne a\in T\), write
		\(a=c^\nu a_\nu(1+\sum_{j\ge1}b_jc^j)\).
		The \(b_j\) have exponentially bounded weights. Its formal reciprocal has the same type of bound: in degree \(n\), each product satisfies
		\(\sum_\ell2^{j_\ell}\le2^n\), with \(\sum j_\ell=n\). Including the fixed weight of \(a_\nu^{-1}\) and shifting shows that \(c^d/a\in T\) for some \(d\). Hence \(T_c\) is a field.
		
		Since \(W_N\) is finite-dimensional, \(W_N[[x]]=W_N\otimes_kF=E_N\). Reordering coefficients identifies
		\(A=\{\sum b_jc^j:b_j\in E_{2^j-1}\}\subseteq H\).
		Every member of \(H\) enters \(A\) after a sufficiently large coefficient shift. Thus \(A_c=H_c\).
	\end{proof}
	
	\begin{lemma}\label{G-regular}
		The ring \(G=A_c\) is regular Noetherian of dimension \(r\). It contains a regular sequence \(a_1,\ldots,a_r\) generating a proper ideal, with \(ca_i\in A\).
	\end{lemma}
	\begin{proof}
		Choose \(\varphi_1,\ldots,\varphi_r\in x\Q[[x]]\) algebraically independent over \(\Q\). Such a choice exists by induction: the field generated by finitely many chosen series over \(\Q\) is countable, as is its algebraic closure inside \(\Q((x))\), whereas \(x\Q[[x]]\) is uncountable. Put \(a_i=u_i-\varphi_i(x)\).
		
		Successive substitution in \(F[u_1,\ldots,u_r]\) shows that the \(a_i\) form a regular sequence. Algebraic independence ensures that the denominators in \(k[u_1,\ldots,u_r]\setminus\{0\}\) have nonzero images under full substitution. Thus the sequence stays regular in \(B_0\), and substitution defines \(B_0\to k((x))\). It extends coefficientwise to
		\(H\to k((x))[[c]]\) and then to \(G\to k((x))((c))\), annihilating every \(a_i\) and preserving \(1\). Their ideal in \(G\) is therefore proper. Also \(u_i\in W_1\), so \(ca_i\in A\).
		
		The ring \(B_0\) is regular Noetherian. For each maximal ideal of \(H\), its local ring contains the regular element \(c\) and has regular local quotient modulo \(c\), by \eqref{structural}. A Noetherian local ring with a regular element and regular quotient is regular: the lifted generators give embedding dimension at most the quotient dimension plus one, which equals its dimension. Hence \(H\) and \(G\) are regular.
		
		The completion \(B_0[[c]]\) is faithfully flat over Noetherian \(H\) and flat over \(B_0\). Flatness descends, so \(H\), and then \(G\), are flat over \(B_0\). The regular sequence stays regular in \(G\). Its properness gives \(\dim G\ge r\), while \eqref{structural} gives the reverse inequality. Thus
		\[
		\gld G=\wgd G=r.
		\]
	\end{proof}
	
	\begin{lemma}\label{truncations}
		One has \(\FPD(T)=\FFD(T)=1\). For every \(j\ge1\), if \(C_j=A/c^jA\), then \(\FPD(C_j)=\FFD(C_j)=1\).
	\end{lemma}
	\begin{proof}
		Because \(\mathfrak m^{3j}\subseteq c^jT\), the local ring \(T/c^jT\) has nilpotent maximal ideal and is perfect. Since \(T_c\) is a field, Theorem \ref{flat-formula} gives \(\FFD(T)=1\), and \eqref{almost-perfect} gives \(\FPD(T)=1\).
		
		Principal coefficient division gives \(C_j=(T/c^jT)[[x]]\). Its nilpotent coefficient radical extends to a nilpotent ideal, with quotient \(k[[x]]\). Thus \((C_j)_x\) is a nilpotent extension of \(k((x))\) and is perfect. Every \(C_j/x^eC_j=(T/c^jT)[x]/(x^e)\) is perfect as well. Apply Theorem \ref{flat-formula} and \eqref{almost-perfect} with \(x\). The module \(C_j/xC_j\) supplies both lower bounds.
	\end{proof}
	
	\begin{proposition}\label{seed-values}
		For every \(r\ge2\),
		\[
		\FPD(T_r)=\FFD(T_r)=1,\qquad
		\FFD(T_r[[x]])=r,\qquad \FPD(T_r[[x]])=r+1.
		\]
		There is a \(G=A_c\)-module \(N\) with \(\pd_A N=r+1\).
	\end{proposition}
	\begin{proof}
		The coefficient values are Lemma \ref{truncations}. Theorem \ref{flat-formula} and Lemma \ref{G-regular} give
		\(\FFD(A)=\max\{2,r\}=r\).
		
		For the projective upper bound, let \(\pd_A M<\infty\) and take \(K=\Omega_A^rM\). It is flat by the flat bound, and \(K_c\) is projective because \(\gld G=r\). If \(Z=\Omega_A M\), each \(c^j\) is regular on \(Z\), so a finite projective resolution of \(Z\) remains exact modulo \(c^j\). Hence \(Z/c^jZ\) has finite \(C_j\)-projective dimension, at most one. The module \(K/c^jK\) is its \((r-1)\)st syzygy, and therefore projective, since \(r\ge2\). Equation \eqref{strong-flat} gives \(\pd_AK\le1\), proving \(\pd_AM\le r+1\).
		
		For the lower bound put
		\[
		b_i=ca_i,\quad C=A/(b_1,\ldots,b_r),\quad
		N=C_c=G/(a_1,\ldots,a_r)G,\quad W=C^{(\N_0)}.
		\]
		We have \(N\ne0\) and \(\pd_GN=r\), so \(\pd_AN\le r+1\). The telescope computes
		\[
		E=\Ext_A^1(G,W)=
		\operatorname{Coker}\left[\prod_{j\ge0}W\longrightarrow\prod_{j\ge0}W,\
		(v_j)\longmapsto(v_j-cv_{j+1})\right].
		\]
		The class of the sequence of standard vectors \((e_j)\) is nonzero. Otherwise choose \(j\) outside the finite support of \(v_0\), iterate the recurrence, and take coordinate \(j\). This gives
		\(0=c^j(1+c(v_{j+1})_j)\) in \(C\). The parenthesis is a unit, because \(C\) is local and \(c\in\Jac(C)\). Thus \(c^j=0\), contradicting \(C_c\ne0\).
		
		The group \(E\) is naturally a \(G\)-module through the first Ext argument; the two scalar actions agree. Since \(b_iW=0\), we have \(b_iE=0\), and then \(a_iE=0\). The Koszul resolution gives \(\Ext_G^r(N,E)=E\ne0\). In
		\[
		\Ext_G^p(N,\Ext_A^q(G,W))\Longrightarrow\Ext_A^{p+q}(N,W),
		\]
		only rows \(q=0,1\) occur. The term \((r,1)\) has no incoming higher differential and an outgoing differential lands beyond \(\pd_GN=r\). Hence \(\Ext_A^{r+1}(N,W)\ne0\) and \(\pd_AN=r+1\).
	\end{proof}
	
	\begin{proposition}\label{zero-flat-seed}
		For \(r\ge2\) let \(W_r=T_r\ideal T_r/cT_r\). Then
		\[
		\FFD(W_r)=0,\quad \FPD(W_r)=1,\quad
		\FFD(W_r[[x]])=r,\quad \FPD(W_r[[x]])=r+1.
		\]
	\end{proposition}
	\begin{proof}
		Corollary \ref{ideal-bounds} gives the coefficient flat value zero because \(T_r/cT_r\) and \((T_r)_c\) have flat finitistic dimension zero. It gives \(\FPD(W_r)\le1\). The nonzero field \((W_r)_c\) has projective dimension one by the telescope and the impossibility of freeness over local \(W_r\).
		
		Coefficientwise, \(W_r[[x]]=A\ideal A/cA\). Corollary \ref{ideal-bounds} bounds its flat invariant by \(\max\{1,r\}=r\), and its localization \(G\) attains this bound. Its projective invariant is at most \(\FPD(A)=r+1\). The \(c\)-invertible module \(N\) of Proposition \ref{seed-values} transfers unchanged through Theorem \ref{idealization}, giving the opposite bound.
	\end{proof}
	
	\subsection{A perfect seed for simultaneous infinite values}
	The infinite pairs admit one construction for both invariants. Fix a field \(k\) and, for \(i\ge1\), put
	\begin{equation}\label{infinite-blocks}
		C_i=k[t_i,y_{1,i},\ldots,y_{i,i}]/(t_i,y_{1,i},\ldots,y_{i,i})^{i+1},
		\qquad \mathfrak m_i=(t_i,y_{1,i},\ldots,y_{i,i})C_i.
	\end{equation}
	Define the unital \(k\)-algebra
	\begin{equation}\label{infinite-seed}
		B_\infty=k\cdot1+\bigoplus_{i\ge1}\mathfrak m_i
		\subseteq\prod_{i\ge1}C_i.
	\end{equation}
	The scalar part is diagonal, the radical parts have finite support, and distinct radical blocks annihilate each other. Thus multiplication is
	\[
	(\lambda,(u_i))(\mu,(v_i))=(\lambda\mu,(\lambda v_i+\mu u_i+u_iv_i)).
	\]
	The notation \(B_\infty\) is reserved for this seed and is independent of the ring \(B\) in Section~4.
	
	\begin{lemma}\label{infinite-perfect}
		The ring \(B_\infty\) is local and perfect, with \(T\)-nilpotent, nonnilpotent maximal ideal
		\(\mathfrak m=\bigoplus_i\mathfrak m_i\). Moreover
		\[
		\Spec B_\infty=\{\mathfrak m\},\qquad
		\FFD(B_\infty)=\FPD(B_\infty)=0.
		\]
	\end{lemma}
	\begin{proof}
		Every element of \(\mathfrak m\) is nilpotent, since its support is finite. If \(\lambda\ne0\), a finite geometric series inverts \(\lambda+u\), for \(u\in\mathfrak m\). Hence \(\mathfrak m\) is precisely the set of nonunits. For any sequence \(a_1,a_2,\ldots\in\mathfrak m\), let \(F\) be the finite support of \(a_1\). Each initial product is supported in \(F\). If \(N>\max F\), its \(i\)th coordinate belongs to \(\mathfrak m_i^N=0\) for every \(i\in F\). The case \(a_1=0\) is immediate. Thus \(\mathfrak m\) is \(T\)-nilpotent. Bass's characterization of perfect rings applies \cite{Bass}. On the other hand \(t_i^i\ne0\), so \(\mathfrak m\) is not nilpotent. Every prime contains the nil ideal \(\mathfrak m\), proving the assertion about the spectrum.
		
		Here is a direct verification of the two homological values. Successive projective covers over a perfect local ring give a minimal projective resolution, whose differentials have image in the maximal ideal times the preceding term. All projectives are free \cite{Kaplansky}. The nonzero element \(s=t_1\), supported in the first block, satisfies \(s\mathfrak m=0\). If a module has finite positive projective dimension \(q\), its minimal resolution ends in a nonzero free module \(P_q\) and an injective differential \(d_q:P_q\to P_{q-1}\) with image in \(\mathfrak mP_{q-1}\). For a basis vector \(e\in P_q\), one has \(se\ne0\) but \(d_q(se)=0\), a contradiction. To justify termination for arbitrary modules, the cover of a projective syzygy splits and has superfluous kernel, hence is an isomorphism. Thus every module of finite projective dimension is projective.
		
		If a module has finite flat dimension, its last syzygy in a free resolution is flat. Perfectness makes this syzygy projective, so the module has finite projective dimension and is itself projective. This proves both zero values.
	\end{proof}
	
	Put
	\[
	A_\infty=B_\infty[[x]],\qquad D_\infty=\prod_{i\ge1}C_i[[x]],
	\qquad f=\sum_{i\ge1}t_ix^i\in xA_\infty.
	\]
	The coordinate maps identify \(A_\infty\) with a subring of \(D_\infty\). At each degree the coordinate coefficients have a common scalar part and only finitely many radical deviations. The set of deviations may vary with the degree. In \(D_\infty\) the \(i\)th coordinate of \(f\) is \(t_ix^i\).
	
	\begin{lemma}\label{infinite-localization}
		One has \(fD_\infty\subseteq A_\infty\), and inclusion induces an isomorphism
		\[
		U_\infty:=(A_\infty)_f\ \cong\ (D_\infty)_f\ne0.
		\]
	\end{lemma}
	\begin{proof}
		For \(h=(\sum_{q\ge0}h_{i,q}x^q)_i\in D_\infty\), the coefficient of \(x^m\) in \(fh\) is supported in the finitely many blocks \(i\le m\) and has \(i\)th coordinate \(t_ih_{i,m-i}\in\mathfrak m_i\). Hence \(fh\in A_\infty\). A fraction \(h/f^q\) in \((D_\infty)_f\) equals \((fh)/f^{q+1}\) and therefore comes from \((A_\infty)_f\). If a fraction from \(A_\infty\) vanishes after mapping to \((D_\infty)_f\), a power of \(f\) kills its numerator in \(D_\infty\) and thus in \(A_\infty\). This proves injectivity. Finally, for every \(N\ge1\) and \(i\ge N\), the \(i\)th coordinate of \(f^N\) is \(t_i^Nx^{iN}\ne0\). The element \(f\) is not nilpotent, so the localization is nonzero.
	\end{proof}
	
	\begin{proposition}\label{infinite-regular}
		For \(j\ge1\), let \(b_j\in D_\infty\) have \(i\)th coordinate \(y_{j,i}\) if \(i\ge j\), and zero otherwise. Every finite initial sequence \(b_1,\ldots,b_r\) is regular in \(U_\infty\).
	\end{proposition}
	\begin{proof}
		A finitely generated ideal in a product is the product of its coordinate ideals. Indeed, coordinatewise expressions using the same finite set of generators collect into finitely many tuples of coefficients. Consequently, with
		\(E_j=D_\infty/(b_1,\ldots,b_{j-1})D_\infty\), we may compute kernels coordinatewise. For \(i\ge j\), its \(i\)th factor is
		\[
		\left(k[t_i,y_{j,i},\ldots,y_{i,i}]/(t_i,y_{j,i},\ldots,y_{i,i})^{i+1}\right)[[x]].
		\]
		Reduction commutes with power series in this calculation: the coefficientwise kernel is generated by the finitely many removed variables, and coefficient expressions can be collected into series.
		
		In a truncated polynomial ring \(k[t,z_1,\ldots,z_q]/(t,z_1,\ldots,z_q)^{\ell+1}\), the annihilator of \(z_1\) is the span of the monomials of total degree \(\ell\). Multiplication by \(z_1\) sends distinct monomials of smaller degree to distinct surviving monomials. Thus multiplication by \(t\) kills this annihilator, and the same statement holds coefficientwise for series. On the coordinates \(i\ge j\), the annihilator of \(b_j\) in \(E_j\) is therefore killed by \(f\). On the coordinates \(i<j\), the power \((t_ix^i)^j\) is zero because \(\mathfrak m_i^{i+1}=0\). We obtain the single uniform estimate
		\begin{equation}\label{infinite-uniform}
			f^j\Ann_{E_j}(b_j)=0.
		\end{equation}
		Exactness of localization now shows that multiplication by \(b_j\) is injective on
		\((E_j)_f=U_\infty/(b_1,\ldots,b_{j-1})U_\infty\).
		
		It remains to prove properness. If \((b_1,\ldots,b_r)U_\infty=U_\infty\), some \(N\ge0\) satisfies
		\(f^N\in(b_1,\ldots,b_r)D_\infty\). Choose \(i>\max\{N,r\}\). After removing \(y_{1,i},\ldots,y_{r,i}\) in that coordinate, the image \(t_i^Nx^{iN}\) is still nonzero: the monomial \(t_i^N\) has degree at most \(i\). This contradicts ideal membership. The ideal is proper, completing the proof.
	\end{proof}
	
	\begin{proposition}\label{infinite-witnesses}
		For every \(r\ge1\), the nonzero module
		\(M_r=U_\infty/(b_1,\ldots,b_r)U_\infty\) satisfies
		\[
		\fd_{A_\infty}M_r=\fd_{U_\infty}M_r=\pd_{U_\infty}M_r=r,
		\qquad r\le\pd_{A_\infty}M_r\le r+1.
		\]
		Furthermore
		\[
		\Tor_i^{A_\infty}(M_r,M_r)\cong M_r^{\binom ri}\quad(0\le i\le r),
		\qquad \Tor_i^{A_\infty}(M_r,M_r)=0\quad(i>r).
		\]
		Each \(M_r\) is countably presented, but is not finitely generated, over \(A_\infty\). In particular
		\[
		\FFD(A_\infty)=\FPD(A_\infty)=\infty.
		\]
	\end{proposition}
	\begin{proof}
		The Koszul complex on the regular sequence of Proposition~\ref{infinite-regular} is a finite free \(U_\infty\)-resolution of \(M_r\). Its exactness follows by induction: adjoining the last element forms the mapping cone of multiplication by that element, and injectivity on the preceding quotient kills the positive homology. This argument does not require Noetherianity. On tensoring with \(M_r\), all differentials vanish, and the top term is \(M_r\ne0\). Hence its flat and projective dimensions over \(U_\infty\) are exactly \(r\).
		
		Lemma~\ref{localization} gives the flat equality over \(A_\infty\) and the finite projective upper bound \(r+1\). Localization of projective resolutions gives the lower bound \(r\). The Koszul complex is also a flat \(A_\infty\)-resolution, since \(U_\infty\) is a localization of \(A_\infty\). The identity \(U_\infty\otimes_{A_\infty}M_r=M_r\) gives the stated Tor computation over \(A_\infty\) itself. Letting \(r\) increase proves that both sets of finite dimensions are unbounded.
		
		For the presentation, put \(a_j=fb_j\in A_\infty\). Since \(f\) is invertible in \(U_\infty\),
		\[
		M_r\cong\bigl(A_\infty/(a_1,\ldots,a_r)A_\infty\bigr)_f.
		\]
		Generators \(e_q\), for \(q\ge0\), and relations
		\[
		e_q-fe_{q+1}=0,\qquad a_je_q=0\quad(q\ge0,\ 1\le j\le r)
		\]
		therefore give a countable presentation. Finally \(f\in xA_\infty\subseteq\Jac(A_\infty)\) and \(fM_r=M_r\). Finite generation would force \(M_r=0\) by Nakayama's lemma, a contradiction.
	\end{proof}
	
	For example, the flat \(A_\infty\)-resolutions for \(r=1\) and \(r=2\) are
	\[
	0\longrightarrow U_\infty\xrightarrow{b_1}U_\infty\longrightarrow M_1\longrightarrow0
	\]
	and
	\[
	0\longrightarrow U_\infty\xrightarrow{\binom{-b_2}{b_1}}
	U_\infty^2\xrightarrow{(b_1\ b_2)}U_\infty\longrightarrow M_2\longrightarrow0.
	\]
	They give \(\Tor_1^{A_\infty}(M_1,M_1)=M_1\ne0\) and
	\(\Tor_2^{A_\infty}(M_2,M_2)=M_2\ne0\). In higher degrees the same Koszul construction applies. No determination of which endpoint in \([r,r+1]\) equals \(\pd_{A_\infty}M_r\) is needed for the finitistic conclusion.
	
	\begin{remark}\label{infinite-structure}
		Neither \(B_\infty\) nor \(A_\infty\) is coherent. Indeed \(\Ann_{B_\infty}(t_1)=\mathfrak m\), and \(\mathfrak m\) is not finitely generated: finitely many generators involve only finitely many blocks, which multiplication cannot enlarge. Coefficientwise,
		\(\Ann_{A_\infty}(t_1)=\mathfrak m[[x]]\). If this annihilator were finitely generated, evaluation at \(x=0\) would make \(\mathfrak m\) finitely generated. Thus these examples complement, rather than strengthen, the coherent-target example of Theorem~\ref{main-coherent}. The ideal \(\mathfrak m[[x]]\) need not equal \(\mathfrak mA_\infty\).
		
		No flatness assertion about \(B_\infty\to A_\infty\) or \(A_\infty\to D_\infty\) is used. The flat extension is \(A_\infty\to U_\infty\). Nor is localization interchanged with an infinite product: \((D_\infty)_f\) is the localization of the product at one element. Coordinatewise nilpotence alone would not prove injectivity after localization; the uniform exponent in \eqref{infinite-uniform} is essential. The modules above establish big finitistic dimensions and give no automatic conclusion about invariants restricted to finitely generated modules.
	\end{remark}
	
	\subsection{Varying the block parameters}
	The same construction permits independent choices of the number of variables, truncation degrees, and positions of the coefficients of \(f\).
	\begin{proposition}\label{infinite-parameters}
		Let \(d_i,q_i,w_i\) be positive integers such that \(d_i\to\infty\), \(q_i\to\infty\), and for each \(N\) only finitely many \(i\) satisfy \(w_i\le N\). Put
		\[
		H_i=k[t_i,y_{1,i},\ldots,y_{q_i,i}]/(t_i,y_{1,i},\ldots,y_{q_i,i})^{d_i+1},
		\qquad B'=k\cdot1+\bigoplus_i\mathfrak n_i\subseteq\prod_i H_i,
		\]
		where \(\mathfrak n_i\) is the maximal ideal of \(H_i\). Then \(B'\) is perfect and local, and
		\[
		\FFD(B')=\FPD(B')=0,\qquad
		\FFD(B'[[x]])=\FPD(B'[[x]])=\infty.
		\]
	\end{proposition}
	\begin{proof}
		The support of the first element of a radical sequence is finite. An initial product longer than all the truncation degrees in that support is zero. Thus the radical is \(T\)-nilpotent and the ring is perfect local. The nonzero monomial \(t_1^{d_1}\) annihilates its maximal ideal, so the minimal-resolution argument of Lemma~\ref{infinite-perfect} proves the zero values.
		
		Set \(A'=B'[[x]]\), \(D'=\prod_iH_i[[x]]\), and \(f'=\sum_i t_ix^{w_i}\). If weights coincide, the coefficient is the finite sum over those indices. The finiteness assumption on the weights makes this a series in \(A'\) and gives \(f'D'\subseteq A'\): in degree \(m\), only the finitely many indices with \(w_i\le m\) contribute. The fraction argument of Lemma~\ref{infinite-localization} yields \((A')_{f'}=(D')_{f'}\ne0\); nonvanishing follows from the unbounded truncation degrees.
		
		Define \(c_j\in D'\) to have coordinate \(y_{j,i}\) if \(q_i\ge j\), and zero otherwise. For fixed \(j\), there are only finitely many exceptional coordinates with \(q_i<j\). After killing \(c_1,\ldots,c_{j-1}\), the annihilator of \(c_j\) in every other coordinate is killed by \(t_i\), by the monomial calculation above. On an exceptional coordinate a power \((f'_i)^{d_i+1}\) is zero. Taking a single exponent at least one and at least all these finitely many \(d_i+1\) kills the entire kernel. Hence each \(c_j\) becomes regular modulo its predecessors after localization.
		
		If \((f')^N\) belonged to \((c_1,\ldots,c_r)D'\), choose \(i\) sufficiently large that \(q_i\ge r\) and \(d_i\ge N\). Its image in the coordinate quotient is \(t_i^Nx^{Nw_i}\ne0\), a contradiction. Every finite initial sequence is therefore regular and proper. Its Koszul quotient has flat dimension \(r\) over \(A'\) and finite projective dimension between \(r\) and \(r+1\), by the same localization argument as in Proposition~\ref{infinite-witnesses}. Both infinite values follow.
	\end{proof}
	
	\begin{remark}
		The conditions \(d_i\to\infty\) and \(q_i\to\infty\) are tail conditions. Mere unboundedness would not ensure that the exceptional coordinates in the kernel calculation form a finite set with bounded nilpotence degrees. The ring in \eqref{infinite-seed} is the case \(d_i=q_i=w_i=i\).
	\end{remark}
	
	\subsection{Finite products and prescribed pairs}
	For \(T=T_1\times T_2\), the central idempotents \((1,0)\) and \((0,1)\) decompose every \(T\)-module as a pair \((M_1,M_2)\). Exactness, flatness, and projectivity are componentwise. A resolution of the pair has length at most \(h\) precisely when both component modules have resolutions of length at most \(h\) of the same kind. Thus either dimension of \((M_1,M_2)\) is the maximum of its component dimensions, omitting zero components. In taking the finitistic supremum, finite-dimensional modules over either factor may be extended by zero over the other factor. This proves both inequalities in
	\begin{equation}\label{product}
		\FPD(T)=\max_i\FPD(T_i),\qquad
		\FFD(T)=\max_i\FFD(T_i).
	\end{equation}
	Taking components of each coefficient gives the natural ring isomorphism \(T[[x]]=T_1[[x]]\times T_2[[x]]\). Let \(P_n=k[z_1,\ldots,z_n]\), with \(P_0=k\). Regularity and Theorem \ref{main-noeth} give both coefficient values \(n\) and both target values \(n+1\).
	
	\begin{proof}[Proof of Theorem \ref{main-pairs}]
		First suppose \(d<\infty\). For flat dimension take
		\[
		S_{n,d}=
		\begin{cases}
			k,&(n,d)=(0,1),\\
			W_d\times P_n,&d\ge2.
		\end{cases}
		\]
		Proposition \ref{zero-flat-seed} and \eqref{product} give values
		\(\max\{0,n\}=n\) and \(\max\{d,n+1\}=d\).
		
		For projective dimension with \(n\ge1\), take
		\[
		R_{n,d}=
		\begin{cases}
			P_n,&d=n+1,\\
			T_{d-1}\times P_n,&d\ge n+2.
		\end{cases}
		\]
		The second case has \(d-1\ge2\), and Proposition \ref{seed-values} gives values \(n\) and \(d\). The field \(k\) realizes \((0,1)\).
		
		For \(d=\infty\), take the common family
		\begin{equation}\label{infinite-pairs}
			R_{n,\infty}=S_{n,\infty}=
			\begin{cases}
				B_\infty,&n=0,\\
				B_\infty\times P_n,&n\ge1.
			\end{cases}
		\end{equation}
		Lemma~\ref{infinite-perfect} and \eqref{product} give both coefficient values \(n\). Proposition~\ref{infinite-witnesses} and the power series product isomorphism give both infinite target values. The structural assertions at \(n=0\) follow from Lemma~\ref{infinite-perfect} and Remark~\ref{infinite-structure}. If \(k\) is countable, each \(C_i\) is countable, being finite-dimensional over \(k\). Its finite-support direct sum is countable, as are \(B_\infty\), \(P_n\), and their finite products. Choosing \(k=\Q\) proves the countability assertion.
	\end{proof}

	\section{Perfect coefficients and the missing value two}
In this section, we will show the FPD pair $(0,2)$ is unattainable. 
	
	\subsection{Separating annihilating ideals}
	\begin{lemma}\label{separation}
		Let \(I\) be a \(T\)-nilpotent, nonnilpotent ideal of a commutative ring. For each \(q\ge1\) there are a finitely generated ideal \(J\subseteq I\) and a nonnilpotent ideal \(K\subseteq I\) such that \(J^q\ne0\) and \(JK=0\). Consequently there are finitely generated ideals \(I_n\subseteq I\) with
		\[
		I_nI_m=0\ (n\ne m),\qquad I_n^{n+1}\ne0.
		\]
	\end{lemma}
	\begin{proof}
		Choose \(a\) so that \(N=I\cap\Ann(a)\) is nilpotent but
		\(I\cap\Ann(ab)\) is nonnilpotent for every \(b\in I\). Start with \(a=1\). Whenever a \(b\) violates the latter condition, replace \(a\) by \(ab\). Each resulting \(a\) is nonzero, since otherwise its relative annihilator would be all of \(I\). An infinite procedure would give nonzero initial products from \(I\), contradicting \(T\)-nilpotence. Thus it terminates.
		
		Choose \(s\ge1\) with \(N^s=0\), and regard \(B=I/N\) as an ideal of the unital ring \(R/N\). It is nonnilpotent, and every element has nonnilpotent annihilator in \(B\):
		\[
		\Ann_B(b+N)=(I\cap\Ann(ab))/N.
		\]
		If this quotient were nilpotent, its numerator would be nilpotent as an extension by nilpotent \(N\).
		
		For every \(t\) we claim that \(B\) contains a finitely generated ambient ideal \(U\) and a nonnilpotent ideal \(V\) with \(U^t\ne0\) and \(UV=0\). Choose finitely many \(b_1,\ldots,b_\ell\) generating \(U_0\) with \(U_0^t\ne0\). If their common annihilator is nonnilpotent, the claim follows. Otherwise choose a shortest subfamily with nilpotent intersection of annihilators. Its length is at least two. The intersection \(U_1\) for all but the last element and the annihilator \(V_1\) of the last element are both nonnilpotent, while \(U_1\cap V_1\) is nilpotent. If \((U_1\cap V_1)^e=0\), then
		\(U_1^eV_1^e=0\). Both powers remain nonnilpotent. Choose finitely many elements of \(U_1^e\) generating an ideal with nonzero \(t\)th power and take \(V=V_1^e\).
		
		Apply the claim with \(t=sq\). Lift generators of \(U\) to an ideal \(\widetilde U\subseteq I\) and let \(\widetilde V\) be the preimage of \(V\). Then \(\widetilde U\widetilde V\subseteq N\). Set
		\(J=\widetilde U^s\) and \(K=\widetilde V^s\). We obtain \(JK=0\), \(J^q\ne0\), and \(K\) nonnilpotent. Repeating the construction inside the successive remaining nonnilpotent ideals, with \(q=n+1\), gives the sequence \(I_n\).
	\end{proof}
	
	\subsection{A Bass module with a monomorphism}
	\begin{lemma}\label{bass-mono}
		Suppose \(I_nI_m=0\) for \(n\ne m\) and \(I_n^{n+1}\ne0\). Put \(A=R[[x]]\) and \(D=A_x\). There is a \(D\)-module \(H\) with
		\[
		\fd_D H=\fd_AH=1,\qquad \pd_AH\le2.
		\]
	\end{lemma}
	\begin{proof}
		Choose \(a_{1,n},\ldots,a_{n+1,n}\in I_n\) whose product is nonzero. Define
		\[
		f_j=\sum_{n\ge j}a_{j,n}x^n,\qquad
		g_j=\sum_{n\ge j}a_{j,n}x^{n^2+n}.
		\]
		These are legitimate formal series. Let
		\[
		L=\varinjlim(A\xrightarrow{f_1}A\xrightarrow{f_2}A\longrightarrow\cdots),
		\qquad u_j=f_ju_{j+1}.
		\]
		It is flat and the telescope gives \(\pd_A L\le1\). Multiplication by \(x\) is injective by flatness and surjective because every \(f_j\in xA\). Thus \(L\) is a \(D\)-module, flat over \(D\) as well. The identity
		\(g_jf_{j+1}=f_jg_{j+1}\) follows from annihilation of distinct ideals and defines
		\[
		\gamma:L\longrightarrow L,\qquad \gamma(u_j)=g_ju_{j+1}.
		\]
		
		To prove injectivity, suppose \(bg_jf_{j+1}\cdots f_r=0\) for some \(r\ge j\), as follows from \(\gamma(bu_j)=0\). Work in \(D\). For every \(i\ge r+1\), multiply by \(a_{r+1,i}\). All terms outside the \(i\)th ideal vanish, and cancellation of the invertible power of \(x\) gives
		\(b\,a_{j,i}\cdots a_{r+1,i}=0\). Summing the coefficientwise expressions now gives \(bf_j\cdots f_{r+1}=0\). This is precisely the finite-stage criterion for \(bu_j=0\) in \(L\).
		
		Set \(H=\operatorname{Coker}\gamma\). The exact sequence
		\(0\to L\xrightarrow{\gamma}L\to H\to0\) gives both flat upper bounds and \(\pd_AH\le2\). To show nonflatness, put \(a=g_1\). We have \(\gamma(u_1)=au_2\), but \(u_1\notin aL\). If \(u_1=abu_j\), choose \(j\ge2\) and a common stage \(r\ge j\). Then
		\[
		f_1\cdots f_{r-1}=g_1b f_j\cdots f_{r-1}.
		\]
		For sufficiently large \(n\ge r\), multiply by \(a_{r,n}\). The left side is the nonzero monomial
		\(a_{1,n}\cdots a_{r,n}x^{(r-1)n}\). If \(b\in x^{-q}A\), the right side has \(x\)-order at least
		\(n^2+(r-j+1)n-q\), which is greater for large \(n\). This is impossible.
		
		Thus \(\gamma:L/aL\to L/aL\) has a nonzero kernel. Tensoring the exact sequence over \(D\) gives \(\Tor_1^D(D/(a),H)\ne0\). Lemma \ref{localization} proves both equalities.
	\end{proof}
	
	\subsection{Detecting a third Ext group}
	The next argument uses a scalar ring that is an injective cogenerator over itself. Both a field and \(\mathbb Z/p^e\mathbb Z\) have this property.
	
	\begin{lemma}\label{third-ext}
		In the setting of Lemma \ref{bass-mono}, suppose \(R\) is an algebra over a field \(B_*\) or over \(B_*=\mathbb Z/p^e\mathbb Z\). There is an \(A\)-module \(V\) with \(\pd_A V=3\).
	\end{lemma}
	\begin{proof}
		Keep the notation of that lemma. Put
		\(p_n=a_{1,n}\cdots a_{n,n}\) and \(O_n=\sum_{i\ne n}I_i\).
		Because \(p_na_{n+1,n}\ne0\) and \(O_nI_n=0\), we have \(p_n\notin O_n\). The injective cogenerator property supplies \(B_*\)-linear maps
		\[
		\ell_n:R\longrightarrow B_*,
		\qquad \ell_n(O_n)=0,\qquad \ell_n(p_n)=\delta_n\ne0.
		\]
		For \(\mathbb Z/p^e\mathbb Z\), extend a nonzero map from the cyclic submodule generated by the image of \(p_n\) in \(R/O_n\). The scalar \(\delta_n\) need not be a unit.
		
		Let \(E=B_*^{\N_{>0}}/B_*^{(\N_{>0})}\). For \(\lambda=(\lambda_n)\) put
		\[
		h_{\lambda,j}=\sum_{n\ge j}\lambda_na_{j,n}x^n,\qquad
		\varepsilon_\lambda(u_j)=h_{\lambda,j}u_{j+1}.
		\]
		Mutual annihilation gives compatibility with the direct system. It also gives the multiplicative identity
		\(h_{\lambda,j}h_{\mu,j+1}=h_{\lambda\mu,j}f_{j+1}\).
		The constant sequence \(1\) acts as the identity, and finite-support sequences act as zero at a sufficiently late stage. Hence \(L\) is an \((A,E)\)-bimodule. These actions commute with \(\gamma\), so \(H\) inherits the right \(E\)-action.
		
		Take \(W=A^{(\N_{>0})}\) with basis \(z_j\). The telescope computes
		\[
		C=\Ext_A^1(L,W)=\operatorname{Coker}\left[
		\prod_{j\ge1}W\xrightarrow{\Delta}\prod_{j\ge1}W\right],
		\qquad \Delta(c)_j=c_j-f_jc_{j+1}.
		\]
		Set \(F_0=1\), \(F_j=f_1\cdots f_j\). For \(b=(b_j)\) form
		\[
		w_n(b)=\sum_{j=1}^{n+1}F_{j-1}b_j.
		\]
		Extract the coefficient of \(x^{n^2}\) from its \(z_{n+1}\) coordinate and apply \(\ell_n\); call the result \(r_n(b)\). If \(b=\Delta(c)\), telescoping gives
		\(w_n(b)=c_1-F_{n+1}c_{n+2}\).
		The first term has zero \(z_{n+1}\) coordinate for all sufficiently large \(n\), and \(F_{n+1}\) has order at least \((n+1)^2\). Thus \(r_n(b)\) is eventually zero. We obtain a well-defined map
		\[
		\rho:C\longrightarrow E,\qquad
		\rho([b])=[(r_n(b))_n].
		\]
		
		We verify its two needed compatibilities. A chain lift of \(\varepsilon_\lambda\) sends the \(j\)th basis vector in each telescope term to \(h_{\lambda,j}\) times the \((j+1)\)st. Its action on \(C\) is therefore
		\([b_j]\mapsto[h_{\lambda,j}b_{j+1}]\).
		After applying \(\ell_n\), all terms in \(O_n\) disappear; the remaining coefficient is \(\lambda_n\) times the old coefficient, except for the contribution of \(b_1\) and a term divisible by \(F_{n+1}\). The former is eventually zero in coordinate \(z_{n+1}\) and the latter has excessive degree. Consequently \(\rho\) is \(E\)-linear.
		
		A chain lift of \(\gamma\) similarly acts by \([b_j]\mapsto[g_jb_{j+1}]\). After \(\ell_n\) is applied, the surviving terms in \(F_{j-1}g_jb_{j+1}\) have order at least \(n^2+jn>n^2\). Thus \(\rho\gamma^*=0\).
		
		Since \(\pd_A L\le1\), the exact sequence defining \(H\) gives
		\[
		C_2=\Ext_A^2(H,W)=\operatorname{Coker}(\gamma^*:C\to C).
		\]
		The map \(\rho\) descends to an \(E\)-linear map \(\overline\rho:C_2\to E\). Let \(c\in C_2\) be the class represented by \(b_j=z_j\). Only \(F_nz_{n+1}\) contributes to the selected coordinate and degree, so
		\begin{equation}\label{delta-class}
			\overline\rho(c)=[(\delta_n)_n],\qquad \delta_n\ne0\text{ for all }n.
		\end{equation}
		In particular \(\pd_A H=2\).
		
		Choose an almost disjoint family of infinite subsets \((S_\alpha)_{\alpha<\mathfrak c}\) of \(\N_{>0}\), with \(\mathfrak c=2^{\aleph_0}\). For example, encode the finite prefixes of infinite binary sequences in a countable set. Their characteristic idempotents \(e_\alpha\in E\) are nonzero and pairwise orthogonal. There is a subset \(\mathcal H\subseteq\mathfrak c\) for which no \(T\subseteq\N_{>0}\) satisfies
		\begin{equation}\label{noseparator}
			S_\alpha\setminus T\text{ finite }(\alpha\in\mathcal H),
			\qquad S_\alpha\cap T\text{ finite }(\alpha\notin\mathcal H).
		\end{equation}
		Indeed there are only \(\mathfrak c\) choices of \(T\), each separating at most one partition, whereas there are \(2^{\mathfrak c}\) partitions.
		
		The ideal \(P=\bigoplus_{\alpha<\mathfrak c}e_\alpha E\) is projective. The module \(E/P\) is flat: it is the filtered colimit of the projective quotients by ideals generated by finitely many orthogonal idempotents. No regularity assumption on \(E\) is required. Set \(V=H\otimes_E(E/P)\). Flatness gives
		\[
		0\longrightarrow\bigoplus_\alpha e_\alpha H
		\longrightarrow H\longrightarrow V\longrightarrow0.
		\]
		Its first two terms have \(A\)-projective dimension at most two, so \(\pd_A V\le3\), and the Ext sequence gives
		\[
		\Ext_A^3(V,W)=
		\operatorname{Coker}\left[C_2\longrightarrow\prod_\alpha e_\alpha C_2\right].
		\]
		Consider the vector with coordinate \(e_\alpha c\) for \(\alpha\in\mathcal H\) and zero otherwise. If it came from \(c'\in C_2\), represent \(\overline\rho(c')\) by \((d_n)\). Equation \eqref{delta-class} would imply that
		\(T=\{n:d_n=\delta_n\}\) satisfies \eqref{noseparator}. On an unselected \(S_\alpha\), \(d_n\) is eventually zero and hence eventually different from nonzero \(\delta_n\); on a selected one it is eventually equal. This contradiction proves \(\Ext_A^3(V,W)\ne0\).
	\end{proof}
	
	\begin{proof}[Proof of Theorem \ref{main-perfect}]
		First suppose \(R\) is local perfect, with maximal ideal \(\mathfrak m\). If \(\mathfrak m\) is nilpotent, \(\mathfrak m[[x]]\) is nilpotent and \(A_x\) is a nilpotent extension of the field \((R/\mathfrak m)((x))\), hence perfect. Every \(A/x^nA\) is perfect. Theorem \ref{flat-formula} gives \(\FFD(A)=1\), and \eqref{almost-perfect} gives \(\FPD(A)=1\) and \(\FFD(A_x)=0\).
		
		If \(\mathfrak m\) is not nilpotent, Lemma \ref{separation} produces the ideals needed in Lemmas \ref{bass-mono} and \ref{third-ext}. If the residue characteristic is zero, every nonzero integer is a unit in \(R\), so \(R\) is a \(\Q\)-algebra. If it is \(p>0\), then \(p\in\mathfrak m\) and \(T\)-nilpotence applied to the constant sequence \(p\) gives \(p^e=0\) for some \(e\). Thus \(R\) is a \(\mathbb Z/p^e\mathbb Z\)-algebra. Lemma \ref{third-ext} gives the dimension-three module in either case, and Lemma \ref{bass-mono} gives \(\FFD(A_x)\ge1\).
		
		A general commutative perfect ring is a finite product of local perfect rings. Formal and Laurent series rings respect this decomposition, and \eqref{product} applies. A nonnilpotent radical has a local factor with nonnilpotent radical; extend its witness module by zero on the other factors.
	\end{proof}
	
	\begin{corollary}\label{no-two}
		There is no commutative ring \(R\) satisfying
		\(\FPD(R)=0\) and \(\FPD(R[[x]])=2\).
	\end{corollary}
	\begin{proof}
		The coefficient equality is equivalent to perfectness, so Theorem \ref{main-perfect} applies.
	\end{proof}
	
\section{Perfect coefficients and set-theoretic independence}

The first part establishes a lower bound valid for arbitrary perfect
coefficients. The second constructs local perfect rings with the
required upper bound. Write $\mathfrak c=2^{\aleph_0}$.

\subsection{The cardinal obstruction}

\begin{lemma}\label{boolean-detector}
Let \(R\) be a commutative local perfect ring whose maximal ideal is not
nilpotent, and put \(A=R[[x]]\). There exist an \(A\)-module \(H\), a countably
generated free \(A\)-module \(W_0\), a scalar ring \(K\), and a class
\(c\in\Ext_A^2(H,W_0)\) with the following properties.
\begin{enumerate}[label=\textup{(\roman*)}]
\item The module \(H\) has a length-two resolution by countably generated free
\(A\)-modules.
\item The ring \(E=K^{\N_{>0}}/K^{(\N_{>0})}\) acts on the right of \(H\),
commuting with \(A\). Here \(K=\Q\) or \(K=\mathbb Z/p^e\mathbb Z\).
\item If \(S\subseteq\N_{>0}\) is infinite and \(e_S\in E\) is its
characteristic idempotent, the restriction of \(c\) to
\(\Ext_A^2(He_S,W_0)\) is nonzero. Finite sets act as zero, and the
idempotents satisfy \(e_Se_T=e_{S\cap T}\).
\end{enumerate}
\end{lemma}
\begin{proof}
Use the construction in Lemmas~\ref{separation}, \ref{bass-mono}, and
\ref{third-ext}. If the residue characteristic is zero, take $K=\Q$;
otherwise take $K=\mathbb Z/p^e\mathbb Z$, where $p^e=0$ in $R$.
Such an $e$ exists by $T$-nilpotence. Retain the modules $L,H$, the
monomorphism $\gamma$, and the right $E$-action from those lemmas.
The telescope resolution of $L$ is
\[
0\longrightarrow A^{(\N_{>0})}\xrightarrow{d}
A^{(\N_{>0})}\longrightarrow L\longrightarrow0,
\qquad d(v_j)=v_j-f_jv_{j+1}.
\]
The map $v_j\mapsto g_jv_{j+1}$ on both free terms lifts $\gamma$,
because $f_jg_{j+1}=g_jf_{j+1}$. Its mapping cone is a length-two
resolution of $H=\operatorname{Coker}\gamma$ by countably generated
free modules. This proves (i).

Take $W_0=A^{(\N_{>0})}$. The proof of Lemma~\ref{third-ext} constructs
an $E$-linear map
\[
\bar\rho:\Ext_A^2(H,W_0)\longrightarrow E
\]
and a class $c$ satisfying
$\bar\rho(c)=[(\delta_j)_j]$, with $\delta_j\ne0$ for every $j$;
see \eqref{delta-class}. Hence $e_Sc\ne0$ whenever $S$ is infinite.
For the inclusion $i:He_S\to H$ and the projection $p:H\to He_S$,
one has $p^*i^*c=e_Sc$. Thus $i^*c\ne0$, proving (iii).
The Boolean identities and the vanishing for finite sets follow from
the right $E$-action, proving (ii).
\end{proof}

\begin{lemma}\label{countable-ext-support}
Let \(H\) have a length-two resolution by countably generated projective
\(A\)-modules. If \(Q\) is a direct summand of \(H\),
\(W=\bigoplus_{t\in T}W_t\), and \(\xi\in\Ext_A^2(Q,W)\), then there is
a countable \(T_0\subseteq T\) such that the image of \(\xi\) under every
coordinate projection to \(\Ext_A^2(Q,W_t)\), \(t\notin T_0\), is zero.
\end{lemma}
\begin{proof}
First take \(Q=H\). A representative of the class is a homomorphism
\(P_2\to W\), where \(P_2\) is the last term of the indicated resolution.
Choose a countable generating set of \(P_2\). Each generator has image in
a finite subsum of \(W\), so the whole homomorphism has image in a
countable subsum. Its coordinate homomorphisms outside that subsum vanish.
For a summand \(Q\), pull the class back along the projection \(H\to Q\),
apply the preceding argument, and then restrict along \(Q\to H\).
The composite is the original class.
\end{proof}

\begin{lemma}\label{countable-grid}
Let \(r\ge1\) be finite and let \(|X_i|=\aleph_i\) for \(1\le i\le r\).
Suppose that, for every \(i\), a map \(F_i\) assigns a countable subset of
\(X_i\) to each element of \(\prod_{j\ne i}X_j\). There exists
\(\alpha\in\prod_{i=1}^rX_i\) such that
\[
\alpha_i\notin F_i((\alpha_j)_{j\ne i})\qquad(1\le i\le r).
\]
The empty product, when \(r=1\), is a singleton.
\end{lemma}
\begin{proof}
For \(r=1\), choose a point of \(X_1\) outside the one countable forbidden
set. For \(r>1\), the union of all values of \(F_r\) has cardinality at
most \(\aleph_{r-1}\): its domain has that cardinality, and each value is
countable. Choose \(\alpha_r\in X_r\) outside the union. Fix this last
coordinate in the other maps and apply induction to \(X_1,\ldots,X_{r-1}\).
The resulting tuple satisfies all the required conditions.
\end{proof}

\begin{theorem}\label{main-obstruction}
	Let $R$ be a  commutative perfect ring with nonnilpotent Jacobson
	radical. If $1\le r<\omega$ and $\aleph_r\le\mathfrak c$, there is an
	$R[[x]]$-module $V_r$ such that
	\[
	\pd_{R[[x]]}V_r=r+2.
	\]
	In particular, $\FPD(R[[x]])\ge r+2$.
\end{theorem}

\begin{proof}
A commutative perfect ring is a finite product of local perfect rings.
At least one factor has nonnilpotent radical. By \eqref{product}, a module
over that factor's power series ring extends by zero with the same
projective dimension. It therefore suffices to treat the local case.
Put \(A=R[[x]]\) and choose \(H,W_0,c,E\) from
Lemma~\ref{boolean-detector}.

We construct \(r\) families of commuting idempotents. On the countable set
\[
\Omega=\bigcup_{m\ge1}\bigl(\{m\}\times(\{0,1\}^m)^r\bigr),
\]
choose, for each \(i\), a family of distinct infinite binary sequences
\((\eta_{i,a})_{a\in X_i}\), where \(|X_i|=\aleph_i\). This is possible
because \(\aleph_i\le\mathfrak c\). Define
\[
S_{i,a}=\{(m,s_1,\ldots,s_r)\in\Omega:
s_i=\eta_{i,a}|_m\}.
\]
Two distinct sets in the same family have finite intersection: beyond the
first level at which their branches differ, there is no common point, and
each earlier level is finite. On the other hand, the intersection of one
chosen set from each of the \(r\) families contains exactly one point at
every level, and is infinite. Identify \(\Omega\) with \(\N_{>0}\), and let
\(e_{i,a}\in E\) be the corresponding characteristic idempotents.
They all commute, and \(e_{i,a}e_{i,b}=0\) for \(a\ne b\).

For \(I\subseteq\{1,\ldots,r\}\) and
\(\alpha_I\in\prod_{i\in I}X_i\), set
\[
H_{I,\alpha_I}=H\prod_{i\in I}e_{i,\alpha_i},\qquad
C_q=\bigoplus_{\substack{I\subseteq\{1,\ldots,r\}\\|I|=q}}
\ \bigoplus_{\alpha_I\in\prod_{i\in I}X_i}H_{I,\alpha_I}.
\]
Thus \(C_0=H\). The differential \(d_q:C_q\to C_{q-1}\) is the alternating
sum of the inclusions obtained by deleting one member of \(I\), with the
directions ordered increasingly. Put
\[
V_r=H\Big/\sum_{i=1}^r\sum_{a\in X_i}He_{i,a}.
\]
There is an exact augmented complex
\begin{equation}\label{boolean-cech}
0\longrightarrow C_r\longrightarrow C_{r-1}\longrightarrow\cdots
\longrightarrow C_1\longrightarrow C_0\longrightarrow V_r
\longrightarrow0.
\end{equation}
To prove exactness, first retain finitely many indices in each direction. Decompose \(H\)
by the finite Boolean algebra generated by the retained commuting
idempotents. On a nonzero summand at most one retained idempotent in each
direction acts as the identity, since those in the same family are
orthogonal. Let \(T\) be its set of active directions. If \(T\ne\varnothing\),
the restricted augmented complex is the augmented simplex complex on
\(T\), with coefficients in that summand, and is contractible by inserting
one fixed member of \(T\). If \(T=\varnothing\), all positive terms vanish
and the map from degree zero to the quotient is an isomorphism.
Thus every finite version is exact. Take the filtered colimit over all
finite choices of indices. The terms become the indicated direct sums,
the quotient becomes \(V_r\), and exactness is preserved. This proves
\eqref{boolean-cech}.

Every \(H_{I,\alpha_I}\) is a direct summand of \(H\). Hence every \(C_q\)
has projective dimension at most two, and \eqref{boolean-cech} gives
\(\pd_AV_r\le r+2\). We prove that equality holds. For any \(A\)-module
\(W\), successive long exact Ext sequences give
\begin{equation}\label{top-ext-cokernel}
\Ext_A^{r+2}(V_r,W)\cong
\operatorname{Coker}\bigl[
\Ext_A^2(C_{r-1},W)\longrightarrow\Ext_A^2(C_r,W)\bigr].
\end{equation}
To see the shifts explicitly, write \(K_0=V_r\) and use the short exact
sequences \(0\to K_{j+1}\to C_j\to K_j\to0\) for \(0\le j\le r-1\),
where \(K_r=C_r\). All Ext groups of \(C_j\) above degree two vanish.
For \(j<r-1\) this gives isomorphisms from degree \(r+2-j\) to degree
\(r+1-j\); the last short exact sequence identifies \(\Ext_A^3(K_{r-1},W)\)
with the cokernel displayed in \eqref{top-ext-cokernel}. For \(r=1\),
the last short exact sequence gives the formula directly.

Let \(X=\prod_{i=1}^rX_i\), and take
\(W=\bigoplus_{\alpha\in X}W_\alpha\), with each \(W_\alpha\) a copy of
\(W_0\). Write \(j_\alpha:W_0\to W\) and
\(\pi_\alpha:W\to W_0\) for its inclusion and projection. Let
\(H_\alpha=H_{\{1,\ldots,r\},\alpha}\), and let
\(c_\alpha\in\Ext_A^2(H_\alpha,W_0)\) be the restriction of \(c\).
The full intersection of the sets defining \(H_\alpha\) is infinite,
so \(c_\alpha\ne0\) by Lemma~\ref{boolean-detector}. Since Ext in the
first argument takes direct sums to products, there is a class
\[
\theta\in\Ext_A^2(C_r,W)=
\prod_{\alpha\in X}\Ext_A^2(H_\alpha,W)
\quad\text{with coordinate }(j_\alpha)_*c_\alpha.
\]

Suppose that \(\theta\) came from \(\zeta\in\Ext_A^2(C_{r-1},W)\).
A summand of \(C_{r-1}\) is specified by a missing direction \(i\) and
a tuple \(\beta\in\prod_{j\ne i}X_j\). Denote the corresponding
coordinate of \(\zeta\) by \(\zeta_{i,\beta}\). Its first argument is a
direct summand of \(H\). Lemma~\ref{countable-ext-support} therefore
provides a countable subset \(T_{i,\beta}\subseteq X\) outside which all
coordinate projections of \(\zeta_{i,\beta}\) vanish. Set
\[
F_i(\beta)=\{a\in X_i:(\beta,a)\in T_{i,\beta}\},
\]
where \((\beta,a)\) denotes the full tuple having coordinate \(a\) in
direction \(i\). These sets are countable. Choose \(\alpha\in X\) by
Lemma~\ref{countable-grid}, so that \(\alpha_i\notin F_i(\alpha_{\widehat i})\)
for every \(i\).

Restrict the equality between the image of \(\zeta\) and
\(\theta\) to the summand \(H_\alpha\) of \(C_r\), and then apply
\((\pi_\alpha)_*\). The left side is a signed sum of the restrictions
of precisely the \(r\) face classes \(\zeta_{i,\alpha_{\widehat i}}\).
Every summand is zero after this projection by the choice of \(\alpha\).
The right side is \(c_\alpha\ne0\), a contradiction. Thus \(\theta\)
has nonzero image in \eqref{top-ext-cokernel}. We have proved
\(\Ext_A^{r+2}(V_r,W)\ne0\) and therefore \(\pd_AV_r=r+2\).
\end{proof}

\subsection{The construction with countable coordinate dependence}

\begin{lemma}\label{coordinate-cardinality}
Let $1\le n<\omega$, put $\kappa=\aleph_n$, and assume
$\mathfrak c\le\kappa$. Then $\kappa^{\aleph_0}=\kappa$.
\end{lemma}
\begin{proof}
For every finite $j\ge0$,
$\aleph_j^{\aleph_0}\le\max\{\aleph_j,\mathfrak c\}$.
The assertion for $j=0$ follows from
$\aleph_0^{\aleph_0}=\mathfrak c$. For $j\ge1$, every countable
sequence in $\aleph_j$ is bounded below $\aleph_j$. Each ordinal below
$\aleph_j$ has cardinality at most $\aleph_{j-1}$. The inductive bound
therefore gives
\[
\aleph_j^{\aleph_0}\le
\aleph_j\cdot\aleph_{j-1}^{\aleph_0}
\le\max\{\aleph_j,\mathfrak c\}.
\]
Take $j=n$ and use the reverse inequality furnished by constant sequences.
\end{proof}

Fix $n\ge1$, let $\kappa=\aleph_n$, and put $X=[\kappa]^{<\omega}$.
For a ring $B$, define the following ring of functions, with pointwise operations:
\[
\mathcal C(B)=\{h:X\to B:\text{ for some countable }S\subseteq\kappa,
\ h(u)\text{ depends only on }u\cap S\}.
\]
Every countable collection of such functions depends on a common countable
set, obtained by taking the union of their coordinate sets. For a fixed
countable $S$, the possible arguments $u\cap S$ form the countable set
$[S]^{<\omega}$. In particular, countable operations on coefficients give
a natural isomorphism
\begin{equation}\label{coordinate-series}
\mathcal C(B)[[x]]\cong\mathcal C(B[[x]]).
\end{equation}

Take $k=\F_2$ and set
\begin{equation}\label{coordinate-perfect-ring}
C_i=k[t]/(t^{i+1}),\qquad J_i=\mathcal C(tC_i),\qquad
R_\kappa=k\cdot1+\bigoplus_{i\ge1}J_i
\subseteq\prod_{i\ge1}\mathcal C(C_i).
\end{equation}

\begin{lemma}\label{coordinate-perfect}
The ring $R_\kappa$ is  local, and perfect, with maximal ideal
$\mathfrak m=\bigoplus_iJ_i$. If $\mathfrak c\le\kappa$, then
$|R_\kappa|=|R_\kappa[[x]]|=\kappa$.
\end{lemma}
\begin{proof}
Each element of $\mathfrak m$ is nilpotent, since it has finite block
support and $J_i^{i+1}=0$. Thus $1+a$ is invertible for every
$a\in\mathfrak m$, and $R_\kappa/\mathfrak m=k$ proves locality.
For a sequence in $\mathfrak m$, the first element has finite block
support. Every sufficiently long initial product vanishes on each of
these blocks, so $\mathfrak m$ is $T$-nilpotent. Bass's criterion proves
perfectness.

For each countable coordinate set $S$, there are at most $\mathfrak c$
functions from $[S]^{<\omega}$ to the finite ring $C_i$. There are at
most $\kappa^{\aleph_0}=\kappa$ choices of $S$, by
Lemma~\ref{coordinate-cardinality}. Hence $|R_\kappa|\le\kappa$.
The elements supported in the first block with values
$u\mapsto t\mathbf1_{\{\alpha\in u\}}$, for $\alpha<\kappa$, are distinct,
so equality holds. Finally
$\kappa\le|R_\kappa[[x]]|\le\kappa^{\aleph_0}=\kappa$.
\end{proof}

Put
\[
A_\kappa=R_\kappa[[x]],\quad D_\kappa=(A_\kappa)_x,\quad
P_\kappa=\prod_{i\ge1}\mathcal C(C_i[[x]]),\quad Q_\kappa=(P_\kappa)_x.
\]
Equation~\eqref{coordinate-series} embeds $A_\kappa$ into $P_\kappa$.

\begin{lemma}\label{coordinate-flat-product}
The rings $P_\kappa$ and $Q_\kappa$ satisfy
$\FFD(P_\kappa)\le1$ and $\FFD(Q_\kappa)\le1$.
\end{lemma}
\begin{proof}
Let $H=k[[t,x]]$ and $L=\prod_{i\ge1}\mathcal C(H)$. We first show that
$L$ is coherent and $\wgd L\le2$. Consider a matrix $L^a\to L^b$.
Its finitely many entries, on all countably many blocks together,
depend on a common countable set $S\subseteq\kappa$. For each
$i\ge1$ and $v\in[S]^{<\omega}$, evaluation gives a matrix
$H^a\to H^b$. Its kernel is a finite second syzygy over the
two-dimensional regular local ring $H$, hence free of rank
$r(i,v)\le a$.

Choose a basis for each of these kernels. Let $e_j\in L$ be the
characteristic function of the pairs $(i,v)$ with $r(i,v)\ge j$.
The chosen $j$th basis vectors, extended by zero where $r(i,v)<j$,
give an isomorphism
\[
\ker(L^a\to L^b)\cong\bigoplus_{j=1}^a e_jL.
\]
To verify surjectivity over $L$, take a kernel vector whose entries depend
on a countable set $T$. Its unique coordinates in the chosen bases
depend only on $u\cap(S\cup T)$, so all these coordinate functions belong
to $L$. Injectivity follows by evaluation. Thus each such kernel is
finite projective. Every finitely presented $L$-module has projective
dimension at most two, and $L$ is coherent. Since arbitrary modules are
filtered colimits of finitely presented modules, $\wgd L\le2$.

The element $\delta=(t^{i+1})_i\in L$ is regular, and
$L/\delta L\cong P_\kappa$. Indeed, reduction is surjective by choosing
representatives for each value on $[S]^{<\omega}$, and division by
$t^{i+1}$ in its kernel preserves the same coordinate set.
For a nonzero $P_\kappa$-module $M$ of finite flat dimension,
Lemma~\ref{shift} now gives
\[
\fd_{P_\kappa}M+1=\fd_LM\le2.
\]
This proves the first bound. Lemma~\ref{localization} proves the second.
\end{proof}

\begin{lemma}\label{coordinate-flat-series}
For the ring in \eqref{coordinate-perfect-ring},
$\FFD(A_\kappa)=1$.
\end{lemma}
\begin{proof}
Write $K=k((x))$ and $J=\mathfrak m[[x]][1/x]$. Then $J$ is an ideal
of $Q_\kappa$, and
\[
D_\kappa=K+J,\qquad D_\kappa/J=K.
\]
To check the ideal assertion, observe that every coefficient of a series
in $J$ has finite block support. The common lower bound on the
$x$-degrees in an element of $Q_\kappa$ ensures that each coefficient
of a product involves only finitely many such supports. Its values
remain radical on every block. Coordinate dependence is preserved by
countable unions. The displayed decomposition follows by reducing every
coefficient modulo $\mathfrak m$.

Let $M$ be a $D_\kappa$-module of finite flat dimension. Projection to the
first block splits $J$ as $J^{[1]}\oplus J^{[>1]}$. The nonzero module
$J^{[1]}$ is annihilated by $J$, because $(tC_1)^2=0$, so it is a
nonzero $K$-vector space. It therefore has a copy of $K$ as a
$D_\kappa$-direct summand. From $0\to J\to D_\kappa\to K\to0$,
$\Tor_j^{D_\kappa}(K,M)$ is a direct summand of
$\Tor_{j+1}^{D_\kappa}(K,M)$ for $j\ge1$. Descending from high degrees
shows that these groups vanish in every positive degree. The same
sequence gives $\Tor_{>0}^{D_\kappa}(J,M)=0$ and an injection
$J\otimes_{D_\kappa}M\to M$.

Choose a $K$-basis $(\bar q_\lambda)_{\lambda\in\Lambda}$ of
$Q_\kappa/J$ with $\bar q_0=1$, and lifts $q_\lambda\in Q_\kappa$,
taking $q_0=1$. There is an exact sequence
\[
0\to J^{(\Lambda\setminus\{0\})}
\xrightarrow{\alpha}D_\kappa^{(\Lambda)}\to Q_\kappa\to0,
\quad
\alpha((j_\lambda))=
\left(-\sum_{\lambda\ne0}j_\lambda q_\lambda,(j_\lambda)_{\lambda\ne0}\right).
\]
The basis property proves exactness. After tensoring with $M$, the
nonzero-index coordinates of $\alpha\otimes M$ are the injective maps
$J\otimes_{D_\kappa}M\to M$. Hence
\begin{equation}\label{coordinate-base-change}
\Tor_j^{D_\kappa}(K,M)=\Tor_j^{D_\kappa}(Q_\kappa,M)=0
\quad(j>0).
\end{equation}

Take a free presentation $0\to Z\to F\to M\to0$.
By \eqref{coordinate-base-change}, tensoring a flat resolution of $M$
with $Q_\kappa$ is exact and gives a finite flat resolution of
$Q_\kappa\otimes_{D_\kappa}M$. Its flat dimension is therefore at most
one by Lemma~\ref{coordinate-flat-product}. Exact base change in the
chosen presentation then shows that
$Q_\kappa\otimes_{D_\kappa}Z$ is flat over $Q_\kappa$.
Also $K\otimes_{D_\kappa}Z$ is flat over $K$.
Equation~\eqref{coordinate-base-change} applies to $Z$ as well. Computing
Tor from a free resolution of $Z$ after either base change gives
\[
\Tor_{>0}^{D_\kappa}(Y,Z)=0
\quad\text{for every $Q_\kappa$-module or $K$-module $Y$}.
\]
For an arbitrary $D_\kappa$-module $Y$, the kernel and cokernel of
$Y\to Q_\kappa\otimes_{D_\kappa}Y$ are annihilated by $J$.
For the kernel, use, for every $j\in J$, the map
$q\otimes y\mapsto(jq)y$ to $Y$. For the cokernel, use $JQ_\kappa\subseteq
D_\kappa$. Both are therefore $K$-modules. Factoring through the image
and applying the two Tor exact sequences proves
$\Tor_1^{D_\kappa}(Y,Z)=0$. Thus $Z$ is flat and
$\FFD(D_\kappa)\le1$.

Perfectness gives $\FFD(R_\kappa)=0$. Theorem~\ref{flat-formula}, applied
to $x$, yields
\[
\FFD(A_\kappa)=\max\{1,\FFD(D_\kappa)\}=1.
\]
\end{proof}

\begin{theorem}\label{coordinate-realization}
Let $1\le n<\omega$ and assume $\mathfrak c\le\aleph_n$. The ring
$R_{\aleph_n}$ in \eqref{coordinate-perfect-ring} is a commutative
local perfect ring of cardinality $\aleph_n$, and
\[
\FFD(R_{\aleph_n}[[x]])=1,\qquad
\FPD(R_{\aleph_n}[[x]])=n+2.
\]
\end{theorem}
\begin{proof}
Write $\kappa=\aleph_n$ and retain the preceding notation. In $P_\kappa$
take $f=(tx^i)_i$ and $g=(x^{i^2})_i$, with constant values on $X$ in
each block. Then $f\in A_\kappa$ and $fP_\kappa\subseteq A_\kappa$,
since the coefficient of degree $m$ involves only blocks $i\le m$.
Consequently
\[
U=(A_\kappa)_f=(P_\kappa)_f.
\]
The element $g$ is regular on $P_\kappa$ and hence on $U$. Put $V=U/gU$.
For each $\alpha<\kappa$, let $e_\alpha\in P_\kappa$ have, in every
block, value $e_\alpha(u)=\mathbf1_{\{\alpha\in u\}}$.

These idempotents form an independent family in $V$. Indeed, take disjoint
finite $F,G\subseteq\kappa$, and let
$e=\prod_{\alpha\in F}e_\alpha\prod_{\beta\in G}(1-e_\beta)$.
At $u=F\in X$ its value is one in every block. If $e$ vanished in $V$,
clearing denominators would give $f^Ne\in gP_\kappa$ for some $N$.
For $i>N$, evaluation at $(i,F)$ would give
$t^Nx^{iN}\in x^{i^2}C_i[[x]]$, although $t^N\ne0$ and $iN<i^2$.
This is impossible. Taking $F=G=\varnothing$ also proves $V\ne0$.

Let $I=\sum_{\alpha<\kappa}e_\alpha V$ and $B=V/I$.
Osofsky's independent-idempotent theorem \cite[Theorem A]{Osofsky} gives
$\pd_V I=n$. Since $n\ge1$, the exact sequence
$0\to I\to V\to B\to0$ gives $\pd_V B=n+1$.
Lemma~\ref{shift} applied to $g$ then gives $\pd_U B=n+2$.
Lemma~\ref{localization}, applied to $f$, yields
\begin{equation}\label{coordinate-lower}
n+2\le\pd_{A_\kappa}B\le n+3<\infty.
\end{equation}

By Lemma~\ref{coordinate-perfect}, $|A_\kappa|=\aleph_n$.
Simson's theorem \cite{Simson} bounds the projective dimension of every
flat $A_\kappa$-module by $n+1$.
Lemma~\ref{coordinate-flat-series} shows that the first free syzygy of
any module of finite projective dimension is flat. Its projective
dimension is therefore at most $n+1$, so
$\FPD(A_\kappa)\le n+2$. Combined with
\eqref{coordinate-lower}, this proves the required equality.
\end{proof}

\subsection{The existence criterion and independence}

\begin{proof}[Proof of Theorem~\ref{main-settheory}]
Fix $1\le n<\omega$. Condition (i)(b) implies (i)(a).
Assume (i)(a), and choose a  commutative perfect ring $R$ with
$\FPD(R[[x]])=n+2$. Its radical is not nilpotent, since
Theorem~\ref{main-perfect}(i) gives value one for a nilpotent radical.
If $\mathfrak c>\aleph_n$, then $\mathfrak c\ge\aleph_{n+1}$.
Theorem~\ref{main-obstruction}, applied with $r=n+1$, gives an
$R[[x]]$-module of finite projective dimension $n+3$, a contradiction.
Thus (i)(a) implies (i)(c).
Conversely, if $\mathfrak c\le\aleph_n$, the ring
$R_{\aleph_n}$ in \eqref{coordinate-perfect-ring} satisfies (i)(b),
by Theorem~\ref{coordinate-realization}. This proves all three
equivalences. For $n=1$, Cantor's theorem gives
$\mathfrak c\ge\aleph_1$, so (i)(c) is precisely CH. Under CH,
condition (i)(c) holds for every finite $n\ge1$.

We prove (ii) as a relative-consistency statement. The standard
constructibility and forcing theorems give, for every fixed finite
$k\ge1$,
\begin{equation}\label{continuum-consistency}
\operatorname{Con}(\mathrm{ZFC})\Longrightarrow
\operatorname{Con}(\mathrm{ZFC}+2^{\aleph_0}=\aleph_k);
\end{equation}
see \cite[Chapters 13--15]{Jech}. For the forcing computation,
start, at the level of relative consistency, with ZFC+GCH and put
$\kappa=\aleph_k$. Force with $\operatorname{Add}(\omega,\kappa)$,
whose conditions are finite partial functions
$\kappa\times\omega\to\{0,1\}$, ordered by extension.
An uncountable family of conditions has an uncountable delta-system
of domains. After restricting to an uncountable subfamily, their
values on the finite root agree, so any two are compatible.
Thus the forcing has the countable chain condition and preserves
cardinals. Its generic coordinate functions give $\kappa$ distinct
reals. Every nice name for a real uses countably many conditions;
there are at most $\kappa^{\aleph_0}=\kappa$ such names in the
GCH ground model. Consequently the extension satisfies
$2^{\aleph_0}=\aleph_k$. The usual relative-consistency formulation
uses the forcing theorem and does not require a transitive set model
of ZFC or any large cardinal assumption.

Now fix $n\ge2$. A model with continuum $\aleph_1$ satisfies both
CH and $\mathsf E_n$, by (i). A model with continuum $\aleph_n$
satisfies $\mathsf E_n$ and $\neg\mathrm{CH}$. A model with continuum
$\aleph_{n+1}$ satisfies $\neg\mathsf E_n$, again by (i).
Applying \eqref{continuum-consistency} with $k=1,n,n+1$ proves
the three consistency implications in (ii). If ZFC proved
$\mathsf E_n$, its consistent extension with $\neg\mathsf E_n$
would be impossible; if ZFC proved $\neg\mathsf E_n$, its consistent
extension with $\mathsf E_n$ would be impossible. Finally, the
extension with $\mathsf E_n+\neg\mathrm{CH}$ proves the stated
nonprovability of $\mathsf E_n\Longrightarrow\mathrm{CH}$.
\end{proof}

\end{document}